\documentclass[12pt]{article} 

\usepackage{amsmath,amssymb,amsfonts,amsthm}
\usepackage{sectsty,secdot,multirow,xcolor,xr,enumerate}
\usepackage{array,epsfig,fancyheadings,rotating,hyperref}  
\usepackage[sectionbib]{natbib}
\usepackage[ruled]{algorithm2e}
\sectionfont{\fontsize{12}{14pt plus.8pt minus .6pt}\selectfont}
\subsectionfont{\fontsize{12}{14pt plus.8pt minus .6pt}\selectfont}

\newtheorem{theorem}{Theorem}
\newtheorem{lemma}{Lemma}

\theoremstyle{definition}
\newtheorem{definition}{Definition}

\newtheorem{assumption}{Assumption}
\def\eop{\hfill {\large $\Box$}}

\def\E{\mathrm{E}}
\def\var{\mathrm{Var}}
\def\cov{\mathrm{Cov}}

\def\Gcal{\mathcal{G}}
\def\Hcal{\mathcal H}

\def\Scal{\mathcal S}
\newcommand{\indep}{\;\, \rule[0em]{.03em}{.67em} \hspace{-.27em}
	\rule[-.02em]{.7em}{.03em} \hspace{-.27em}
	\rule[0em]{.03em}{.67em}\;\,}
\newcommand{\trans}{^{\mbox{\tiny {\sf T}}}}

\begin{document}

%
%

\renewcommand{\thefootnote}{}
$\ $\par

\fontsize{12}{14pt plus.8pt minus .6pt}\selectfont \vspace{0.8pc}
\centerline{\large\bf SLICED INVERSE REGRESSION IN METRIC SPACES} 
\vspace{.4cm} 
\centerline{Joni Virta$^1$, Kuang-Yao Lee$^2$, and Lexin Li$^3$}
\vspace{.4cm} 
\centerline{\it $^1$University of Turku}
\centerline{\it $^2$Temple University}
\centerline{\it $^3$University of California at Berkeley}
 \vspace{.55cm} \fontsize{9}{11.5pt plus.8pt minus.6pt}\selectfont

\begin{quotation}
\noindent {\it Abstract:}
In this article, we propose a general nonlinear sufficient dimension reduction (SDR) framework when both the predictor and response lie in some general metric spaces. We construct reproducing kernel Hilbert spaces whose kernels are fully determined by the distance functions of the metric spaces, then leverage the inherent structures of these spaces to define a nonlinear SDR framework. We adapt the classical sliced inverse regression of \citet{Li:1991} within this framework for the metric space data. We build the estimator based on the corresponding linear operators, and show it recovers the regression information unbiasedly. We derive the estimator at both the operator level and under a coordinate system, and also establish its convergence rate. We illustrate the proposed method with both synthetic and real datasets exhibiting non-Euclidean geometry.

\vspace{9pt}
\par
\end{quotation}\par

\fontsize{12}{14pt plus.8pt minus .6pt}\selectfont

\section{Introduction}
\label{sec:intro}

High-dimensional data are now prevailing in almost every branch of science and business, whereas dimension reduction plays a central role in the analysis of such data. A  particularly useful reduction paradigm is \emph{sufficient dimension reduction} (SDR), which embodies a family of methods that aim to reduce the dimensionality while losing no information in a regression setting. Since the pioneering work of \emph{sliced inverse regression} \citep[SIR]{Li:1991}, SDR has enjoyed a rapid development in the past three decades. For a univariate response $Y$ and a $p$-dimensional predictor $X$, SDR seeks a low-dimensional representation, usually in the form of linear combinations $\beta\trans X$, for a $p \times d$ matrix $\beta = (\beta_1, \ldots, \beta_d)$ with $d \leq p$, such that,
\begin{align}\label{eq:linear_sdr}
Y \indep X \mid \beta_1\trans X, \ldots, \beta_d\trans X.
\end{align}
As such, $\beta\trans X$ contains full regression information of $Y$ given $X$, and since $d$ is often much smaller than $p$, dimension reduction is achieved. SDR then seeks the minimum subspace spanned by $\beta$, called the \emph{central subspace}, which uniquely exists under very mild conditions \citep{Yin:2008}. Originating from SIR \citep{Li:1991},  there has been a large body of methods proposed for SDR, mostly in a model-free fashion that does not impose any specific parametric form for the association between $Y$ and $\beta\trans X$. Examples include \citet{Cook:1991, Li:1992, Cook:2002, Xia:2002, Li:2007, Ma:2012, Ma:2013}, among many others. See also \citet{Libook} for a comprehensive review. 

SDR in \eqref{eq:linear_sdr} achieves \emph{linear dimension reduction}, as the low-dimensional representation takes the form of linear combinations of $X$. It preserves the original coordinates of $X$ and is easier to interpret; nevertheless, it is less flexible. A more recent line of SDR research instead seeks \emph{nonlinear dimension reduction} \citep{Fukumizu:2004,Fukumizu:2009,Li:2011, Lee:2013,Li:2017}, such that, 
\begin{align}\label{eq:nonlinear_sdr}
Y \indep X \mid f_1(X),\ldots, f_d(X),
\end{align}
where $ f_1, \ldots, f_d$ are some functions in a Hilbert space. Nonlinear SDR is more flexible, and may require a smaller number of functions than its linear counterpart to capture the full regression information, though it is generally harder to interpret. 

Despite the substantial progress of SDR, most existing SDR solutions target data that reside in a Euclidean space. However, modern data objects are becoming increasingly complex, and often reside in non-Euclidean spaces. Such data are routinely collected in a wide range of applications, such as medical imaging, computational biology, and computer vision, and it is of great interest to understand associations among those complex data objects \citep{LinZhu2017, CorneaZhu2017, Dubey2019, Petersen2019, lin2019nonlinear, pan2020ball}. As examples, we consider geometric data, positive definite matrix data, and compositional data in this article. For instance, in applications such as brain structural and functional connectivity analysis \citep{Zhu2009, Zhang2020}, the data usually come in the form of positive definite matrices, which measure the connectivity strengths of pairs of nodes of a network and admit a certain manifold structure. In applications such as chemistry, geology, and microbiome analysis \citep{LuLi2019}, the data are the proportions of individual components that sum to a fixed constant. Meanwhile, there are many other examples of complex object data \citep{WangMarron2007}. In all these examples, the data reside in some non-Euclidean spaces, and a proper metric is needed in each case to characterize the intrinsic features of the data. 

In this article, we propose a general nonlinear SDR framework when both the predictor and response lie in some general, and possibly different, metric spaces. Our key idea is to construct a pair of reproducing kernel Hilbert spaces (RKHS), whose kernels are fully determined by the distance functions of the metric spaces. We then leverage the inherent structures of these spaces to define a nonlinear SDR framework for the metric space data, and further adapt sliced inverse regression of \citet{Li:1991} within this framework. We build the estimator based on some linear operators, and show it recovers the regression information unbiasedly. We derive the estimator at both the operator level and under a coordinate system. We also establish the convergence rate of the estimator under both settings when the response lies on a general metric space, and when the response is categorical. We illustrate the proposed method with both synthetic and real datasets exhibiting non-Euclidean geometry. 

Our proposal is related to but also clearly differs from the nonlinear SDR method of \citet{Lee:2013}, and some recent SDR solutions involving functional or non-Euclidean data such as \citet{yeh2008nonlinear, Li:2017, tomassi2019sufficient, ying2020fr, LeeLi2022}. In particular, \citet{Lee:2013} developed a general framework for nonlinear SDR and proposed to estimate the functions $f_1, \ldots , f_d $ in \eqref{eq:nonlinear_sdr} as the eigenfunctions of some linear operator defined on a Hilbert space $\Hcal$, but they still targeted the Euclidean data. Besides, they took $\Hcal$ to be an $L_2$-space at the population level and an RKHS at the sample level. Our framework is similar to theirs, but we aim at data residing in a general metric space. Moreover, we take $\Hcal$ to be an RKHS at both the population and sample levels, which makes the connection between the population and sample versions of the estimation procedure more transparent. \cite{yeh2008nonlinear} proposed kernel SIR under the framework of \eqref{eq:nonlinear_sdr}, but required a functional version of the linearity condition. We instead adopt a general form of conditional independence based on $\sigma$-fields and avoid relying on the linearity condition. \citet{Li:2017} considered nonlinear SDR for functional data, where $X$ is a function residing in some Hilbert space.  Relatedly, \citet{LeeLi2022} studied linear SDR when both $X$ and $Y$ are functions in some Hilbert space. By contrast, we consider more general data objects than functional data.  \citet{tomassi2019sufficient} developed linear SDR for compositional data, but restricted to a specific set of parametric models for the conditional distribution of $X$ given $Y$. \citet{ying2020fr} developed SDR when the response is in a metric space and the predictors reside in a Euclidean space. Since the dimension reduction is to be performed for the predictors, our method differs considerably from that of \citet{ying2020fr}. 

The rest of the article is organized as follows. Section \ref{sec:metric_sdr} develops the general framework of nonlinear SDR for data in metric spaces, and Section \ref{sec:msir} derives the metric version of SIR under this framework. Section \ref{sec:finite_sample} describes the finite-sample implementation, and Section \ref{sec:asymptotics} studies the convergence properties of the estimator. Section \ref{sec:examples} presents the numerical studies, and Appendix \ref{sec:proofs} collects all the technical proofs.

\section{Nonlinear SDR for Metric Space Data}
\label{sec:metric_sdr}

In this section, we propose a general framework for conducting nonlinear SDR for data residing in arbitrary metric spaces. It involves three main steps: defining a minimal $\sigma$-field that captures full regression information, constructing reproducing kernel Hilbert spaces of both $X$ and $Y$ from the metric spaces, and using the RKHSs to define a representation of the minimal $\sigma$-field that is easier to estimate. 

Let $(\Omega, \mathcal{F}, P)$ be a complete probability space. Let $(\Omega^0_X, d_X)$, $(\Omega^0_Y, d_Y) $ be arbitrary separable metric spaces, in which the predictor and the response, respectively, take values. We make no further assumption on the data space and, depending on $\Omega^0_X$, $\Omega^0_Y$, there may be multiple feasible choices for the metrics $d_X$, $d_Y$. See, for instance, Section \ref{sec:examples} where we take $\Omega^0_X$ to be some manifold spaces and consider different choices of metrics for $d_X$. 

Let $\mathcal{F}_X$ and $\mathcal{F}_Y$ be the Borel $\sigma$-fields generated by the open sets in the metric topology in $\Omega^0_X$ and $\Omega^0_Y$, respectively. Consider $X : \Omega \rightarrow \Omega^0_X$ that is a $\mathcal{F}/\mathcal{F}_X$-measurable random variable with the distribution $P_X = P \circ X^{-1}$, and $Y : \Omega \rightarrow \Omega^0_Y$ that is a $\mathcal{F}/\mathcal{F}_Y$-measurable random variable with the distribution $P_Y = P \circ Y^{-1}$. For simplicity, suppose the joint random variable $(X, Y)$ is $\mathcal{F}/(\mathcal{F}_X \times \mathcal{F}_Y)$-measurable. Let $P_{X \mid Y}: \mathcal{F}_X \times \Omega^0_Y \rightarrow \mathbb{R} $ be the conditional distribution of $X$ given $Y = y$, and suppose the set $\{ P_{X \mid Y}( \cdot \mid y ) \mid y \in \Omega^0_Y \} $ is dominated by a $\sigma$-finite measure. Let $\sigma_X$ be the $\sigma$-field generated by $X$. We adopt the following definition from \citet{Lee:2013}. 

\begin{definition}\label{def:sdr_field}
A sub-$\sigma$-field $\mathcal{G}$ of $\sigma_X$ is said to be a sufficient dimension reduction $\sigma$-field for $Y$ given $X$, if the random elements $Y$ and $X$ are conditionally independent given $\mathcal{G}$, in that $Y \indep  X \mid \mathcal{G}$. When the set of conditional distributions, $\{ P_{X \mid Y}( \cdot \mid y ) \mid y \in \Omega^0_Y \} $ is dominated by a $ \sigma$-finite measure, the intersection of all SDR $\sigma$-fields is itself an SDR $\sigma$-field. It is called the central $\sigma$-field, and denoted by $\mathcal{G}_{Y \mid X}$.
\end{definition}

\noindent 
Definition \ref{def:sdr_field} suggests that there exists uniquely a smallest SDR $\sigma$-field. In our pursuit of nonlinear SDR, we seek a set of functions $f_1, \ldots , f_d$ lying in some suitable function space $\Hcal_X$ that are $\mathcal{G}_{Y \mid X}$-measurable, and dimension reduction is achieved by replacing $X$ with the corresponding sufficient predictors $f_1(X), \ldots , f_d(X)$. 

A natural candidate for the function space $\Hcal_X$ is $L_2(P_X)$, the class of all square integrable functions $f: \Omega^0_X \rightarrow \mathbb{R}$ and, indeed, this is what \cite{Lee:2013} used. We instead take $\Hcal_X$ to be a suitably defined reproducing kernel Hilbert space, a choice that makes the subsequent methodology and theory development considerably simpler. More specifically, to connect the RKHS $\Hcal_X$ to the metric structure of the space $\Omega^0_X$, we consider a positive semi-definite kernel, $\kappa_X: \Omega^0_X \times \Omega^0_X \rightarrow \mathbb{R}$, for which there exists a function $\rho: \mathbb{R} \rightarrow \mathbb{R}$, such that, for all $x_1, x_2 \in \Omega^0_X$, 
\begin{align} \label{eq:kernel_metric}
\kappa(x_1, x_2) = \rho\{ d_X(x_1, x_2) \}, 
\end{align}
where $d_X$ is the metric of $\Omega^0_X$. We further impose the following finite second-order moment requirement for the kernel function, which is essentially the RKHS-equivalent of requiring a random variable to be square integrable, and is a rather mild condition. 

\begin{assumption}\label{assu:L2_subset}
Suppose $\mathrm{E} \{ \kappa_X (X, X) \} < \infty$, and $\mathrm{E} \{ \kappa_Y (Y, Y) \} < \infty$.
\end{assumption}

\noindent
There are multiple choices for this type of kernel function, for instance, the Gaussian kernel and the Laplace kernel, among others. Throughout our implementation, we employ the Gaussian kernel with a positive covariance.

Given the kernels $\kappa_X$ and $\kappa_Y$, let $\Hcal^0_X$ and $\Hcal^0_Y$ be the RKHSs generated by $\kappa_X$ and $\kappa_Y$, respectively. By Assumption \ref{assu:L2_subset}, we have that $\Hcal^0_X \subseteq L_2(P_X)$ and $\Hcal^0_Y \subseteq L_2(P_Y)$. Moreover, by the Riesz representation theorem, there exist a unique mean element $\mu_X \in \Hcal^0_X$, and a unique covariance operator $\Sigma_{XX}^0$, such that, 
\begin{align*}
\begin{split}
\langle f, \mu_X \rangle_{\Hcal^0_X} & = \E \{ f (X) \}, \;\; \textrm{ for all } \;\; f \in \Hcal^0_X, \\
\langle f, \Sigma_{XX}^0 f' \rangle_{\Hcal^0_X} & = \cov \{ f(X), f'(X) \}, \;\; \textrm{ for all } \;\;  f, f' \in \Hcal^0_X.  
\end{split}
\end{align*}
Note that every $f_0 \in \mathrm{ker}(\Sigma_{XX}^0)$ satisfies that $\var\{ f_0(X) \} = \langle f_0, \Sigma_{XX}^0 f_0 \rangle_{\Hcal^0_X} = 0$, and is almost surely equal to a constant, where $\mathrm{ker}(\cdot)$ denotes the null space. As such, we further restrict our attention to $\Hcal_X = \overline{\mathrm{ran}}(\Sigma_{XX}^0)$, where $\mathrm{ran}(\cdot)$ denotes the range, and $\overline{\mathrm{ran}}(\cdot)$ denotes the closure of the range. 

\begin{lemma}\label{lem:almost_sure_representer}
Suppose Assumption \ref{assu:L2_subset} holds. There exists a set $\Omega_X \subseteq \Omega^0_X$, such that $P_X(\Omega_X) = 1$, and $\kappa_X( \cdot , x) - \mu_X \in \Hcal_X$ for all $x \in \Omega_X$.
\end{lemma}

\noindent
Lemma \ref{lem:almost_sure_representer} reveals that the functions $\kappa_X( \cdot , x) - \mu_X$, for $x \in \Omega_X$, belong to the space $\Hcal_X$, which allows us to perform centering through the inner product, $\langle f, \kappa_X( \cdot , x) - \mu_X \rangle_{\Hcal_X} = f(x) - \E \{ f(X) \}$. Its proof also shows that the space $\Hcal_X$ admits an alternative characterization, i.e., $\Hcal_X = \overline{\mathrm{span}}\{\kappa_X( \cdot , x) - \mu_X : x \in \Omega_X \}$, where $\overline{\mathrm{span}}(\cdot)$ denotes the closure of the say spanned by the set of functions. We briefly remark that a similar result was obtained in \citet[Lemma~1]{Li:2017}. However, their proof implicitly assumed that the only set for which $P_X$ assigns a zero probability is the empty set, essentially ruling out all continuous distributions, whereas our Lemma~\ref{lem:almost_sure_representer} fixes this issue. We also remark that, this characterization does not imply that the elements $f \in \mathcal{H}_X$ are centered in the sense that $\mathrm{E}\{f(X)\} =0$. Instead, focusing on $\Hcal_X$ removes the constant functions that are of no interest in our dimension reduction pursuit.  We construct $\mu_Y$, $\Sigma^0_{YY}$,  and the RKHS $\Hcal_Y$ in an analogous manner.

\begin{definition}\label{def:central-class}
We call the set of all $f \in \Hcal_X$ that are $\Gcal_{Y \mid X}$-measurable the \textit{central class}, and denote this set by $\Scal_{Y \mid X}$. 
\end{definition}

\noindent
We make two remarks.  First, our notion of dimension reduction is based on the smallest SDR $\sigma$-field, i.e., the central $\sigma$-field. In our setting, the concept of ``dimensionality" is less obvious than that in the classical SDR setting, which is simply the dimension of the central subspace. This is because there are sets that generate the same $\sigma$-field, but with very different dimensions. Nevertheless, our formulation is useful when one is interested in reducing the dimensionality in the class sense, as the central class induced by the central $\sigma$-field contains all such sets of functions generating the same $\sigma$-field, and we seek the smallest one.  Second, the relation between the central $\sigma$-field $\Gcal_{Y \mid X}$ and the central class $\Scal_{Y \mid X}$ is analogous to the relation between the central subspace and the sufficient predictors in the classical setting. That is, in lieu of estimating $\Gcal_{Y \mid X}$, we search for subsets of elements of $\Scal_{Y \mid X}$, which are more concrete and easier to estimate.

\section{Metric Sliced Inverse Regression}
\label{sec:msir}

In this section, we derive the population-level sliced inverse regression for metric space data. Recall the classical SIR \citep{Li:1991} when both $X$ and $Y$ lie in an Euclidean space. It estimates the central subspace by the range of the matrix,
\begin{align}\label{eq:SIR_matrix}
\var(X)^{-1} \var\{ \E( X \mid Y) \},
\end{align}
We next derive the operator analogue for \eqref{eq:SIR_matrix} for two cases: the general case of $Y$ residing in a metric space, and the special case of $Y$ being a discrete random variable.

\subsection{Metric response}

We first define a number of covariance operators that serve as the main building blocks of our nonlinear metric SIR procedure.
\begin{eqnarray}\label{eq:three_covariance_operators}
&\Sigma_{XX} : \Hcal_X \rightarrow \Hcal_X, & \langle f, \Sigma_{XX} f' \rangle_{\Hcal_X} = \cov \{ f(X), f'(X) \}, \nonumber \\
&\Sigma_{XY} : \Hcal_Y \rightarrow \Hcal_X, & \langle f, \Sigma_{XY} g \rangle_{\Hcal_X} = \cov \{ f(X), g(Y) \}, \\
&\Sigma_{YY} : \Hcal_Y \rightarrow \Hcal_Y, & \langle g', \Sigma_{YY} g \rangle_{\Hcal_Y} = \cov \{ g'(Y), g(Y) \}, \nonumber
\end{eqnarray}
for $f, f' \in \Hcal_X$ and $g, g' \in \Hcal_Y$. In addition, the cross-covariance operator $\Sigma_{YX} : \Hcal_X \rightarrow \Hcal_Y$ can be obtained as $\Sigma_{YX} = \Sigma_{XY}^*$, the adjoint of the operator $\Sigma_{XY}$. We also note that, because $\Hcal_X = \overline{\mathrm{ran}}(\Sigma^0_{XX})$, we have $\mathrm{ker}(\Sigma_{XX}) = \{ 0 \}$, and $\overline{\mathrm{ran}}(\Sigma_{XX}) = \Hcal_X$.

We next introduce two regularity conditions. 

\begin{assumption}\label{assu:L2_dense_1}
Suppose that $\Hcal_{X} + \mathbb{R}$ and $\Hcal_{Y} + \mathbb{R}$ are dense in $L_2(P_X)$ and $L_2(P_Y)$, respectively, where $+$ denotes the direct sum. 
\end{assumption}

\begin{assumption}\label{assu:inclusion}
Suppose $\mathrm{ran}(\Sigma_{YX}) \subseteq \mathrm{ran}(\Sigma_{YY})$, and $\mathrm{ran}(\Sigma_{XY}) \subseteq \mathrm{ran}(\Sigma_{XX})$.
\end{assumption}

\noindent
Assumption \ref{assu:L2_dense_1} is typical in kernel learning and generally holds, e.g., when $\kappa_X$ is a Gaussian kernel \citep{Fukumizu:2009}. In this assumption, by ``dense'' we mean that, for every $ f \in L_2(P_X)$, there exists a sequence of elements $f_n \in \Hcal_{X}$, such that $\mathrm{var}\{ f(X) - f_n(X) \} \rightarrow 0$, as $n \rightarrow \infty$.  Assumption \ref{assu:inclusion} is essentially a smoothness condition on the relation between $X$ and $Y$ \citep{li2018linear}, and similar conditions are commonly imposed in SDR \citep{ying2020fr, li2021dimension}. It guarantees that the operator $\Sigma_{YY}^\dagger \Sigma_{YX}$ is both well-defined and bounded \citep[Theorem 1]{douglas1966majorization}, where $\dagger$ denotes the Moore-Penrose pseudo-inverse of $\Sigma_{YY}$; see \cite{li2018linear} for more details on the Moore-Penrose pseudo-inverse of an operator. 

The next lemma provides some useful expressions for the conditional moments of $X$ given $Y$ at the operator level. They are essential to construct the operator analogue for the SIR estimator \eqref{eq:SIR_matrix}. In addition, they help turn conditional moments into unconditional ones, avoiding the slicing step in the original SIR.

\begin{lemma}\label{lem:sir_operator}
Suppose Assumptions \ref{assu:L2_subset}, \ref{assu:L2_dense_1} and \ref{assu:inclusion} hold. Then, 
\begin{enumerate}[(a)]
\item For any $f \in \Hcal_X$, $\E \{ f(X) \mid Y \} - \E \{ f(X) \} = \langle \Sigma_{YY}^\dagger \Sigma_{YX} f, \kappa_Y( \cdot, Y ) - \mu_Y \rangle_{\Hcal_Y}$;

\item For any $f, f' \in \Hcal_X$, $\mathrm{Cov} [ \E \{ f(X) \mid Y \},  \E \{ f'(X) \mid Y \} ] = \langle f, \Sigma_{XY} \Sigma_{YY}^{\dagger} \Sigma_{YX} f' \rangle_{\Hcal_X}$. 
\end{enumerate}
\end{lemma}

By Lemma \ref{lem:sir_operator}, the operator $\Sigma_{XY} \Sigma_{YY}^{\dagger} \Sigma_{YX}$ can be seen as the analogue of the matrix $\var\{ \E( X \mid Y) \}$ in \eqref{eq:SIR_matrix}. Besides, the operator $\Sigma_{XX}^{\dagger}$ can be seen as the analogue of $\var(X)^{-1}$ in \eqref{eq:SIR_matrix}. Consequently, a direct operator counterpart of \eqref{eq:SIR_matrix} is, 
\begin{align}\label{eq:sir_operator_definition}
\Lambda_{\mathrm{SIR}} = \Sigma_{XX}^{\dagger} \Sigma_{XY} \Sigma_{YY}^{\dagger} \Sigma_{YX}.
\end{align}
This operator is well-defined by Assumption \ref{assu:inclusion}. Moreover, it is interesting to note that, if we choose linear kernels $\kappa_X, \kappa_Y$, then $\Lambda_{\mathrm{SIR}}$ reduces precisely to the matrix of the canonical correlation analysis (CCA). 

The next theorem shows that the operator $\Lambda_{\mathrm{SIR}}$ is bounded, and that the closure of its range is unbiased for the central class, which is parallel to the classical SIR for linear SDR of Euclidean data. We need an additional regularity condition.

\begin{assumption}\label{assu:L2_dense_3}
Suppose the set $ \mathrm{ran}(\Sigma_{XX}) \cap \Scal_{Y \mid X}^\bot $ is dense in the set $\Scal_{Y \mid X}^\bot$, where the orthogonal complement is taken with respect to $\Hcal_X$. 
\end{assumption}

\noindent
Assumption \ref{assu:L2_dense_3} requires that the intersection between $\mathrm{ran}(\Sigma_{XX})$ and $\Scal_{Y \mid X}^\bot$ is suitably rich in $\Scal_{Y \mid X}^\bot$, which is a mild condition, since $\mathrm{ran}(\Sigma_{XX})$ is, by definition, dense in its closure $\Hcal_X$. Similar condition has been imposed implicitly in \cite{Li:2017}.

\begin{theorem}\label{theo:SIR_range_closure}
Suppose Assumptions \ref{assu:L2_subset} to \ref{assu:L2_dense_3} hold. Then $\Lambda_{\mathrm{SIR}}$ is a bounded operator, and $\overline{\mathrm{ran}} ( \Lambda_{\mathrm{SIR}} ) \subseteq \Scal_{Y \mid X}$. 
\end{theorem}

\noindent
Theorem \ref{theo:SIR_range_closure} suggests that we can recover the central class by the range of $\Lambda_{\mathrm{SIR}}$, or equivalently, by the spectral decomposition of $\Lambda_{\mathrm{SIR}} \Lambda_{\mathrm{SIR}}^*$. This is the foundation for our estimation procedure developed in Section \ref{sec:finite_sample}. We call our proposed nonlinear SDR method based on $\Lambda_{\mathrm{SIR}}$ as \emph{metric sliced inverse regression} (MSIR).

\subsection{Discrete response}
\label{sec:discrete_response}

Next, we consider a special case when $Y$ lies in the usual Euclidean space and is discrete. This is the scenario that is perhaps most often encountered in real applications. The main difference between this special case and the general case is that, when $Y$ is discrete, we can obtain direct RKHS representations for the conditional moments, instead of resorting to the unconditional representations as in Lemma \ref{lem:sir_operator}.

More specifically, suppose $\Omega^0_Y = \{ 1, \ldots , K \}$, and let $\pi_k = P(Y = k)$, $\pi_k > 0$ for all $k \in \Omega^0_Y$. By the Riesz representation theorem, there exists the elements $\gamma_{X \mid k} \in \Hcal_X$, $k = 1, \ldots, K$, such that, for any $f \in \Hcal_X$, 
\vspace{-0.05in}
\begin{align*}
\mathrm{E} \{ f(X) \mid Y = k \} - \mathrm{E} \{ f(X) \} = \langle \gamma_{X \mid k}, f \rangle_{\Hcal_X} ,
\end{align*}
The elements $\gamma_{X \mid k}$ can be seen to provide a discrete counterpart of Lemma \ref{lem:sir_operator}(a). We then define the covariance operator,
\begin{eqnarray}\label{eq:discrete_operator}
\Gamma_{XX \mid Y} = \sum_{k=1}^K \pi_k (\gamma_{X \mid k} \otimes \gamma_{X \mid k}) : \Hcal_X \rightarrow \Hcal_X, 
\end{eqnarray}
where $\otimes$ denotes the tensor product. It satisfies that, for any $f, f' \in \Hcal_X$, 
\begin{align*}
\cov [ \mathrm{E} \{ f(X) \mid Y \},  \mathrm{E} \{ f'(X) \mid Y \} ] = \langle f, \Gamma_{XX \mid Y} f' \rangle_{\Hcal_X},
\end{align*}
Consequently, the counterpart of $\Lambda_{\mathrm{SIR}}$ in \eqref{eq:sir_operator_definition} when $Y$ is categorical is, 
\begin{align} \label{eqn:msir-discrete}
\Lambda_{\mathrm{SIR, D}} = \Sigma_{XX}^\dagger \Gamma_{XX \mid Y}.
\end{align}
This operator is well-defined under the following smoothness condition, and the closure of its range provides an unbiased estimator of the central class. 

\begin{assumption}\label{assu:inclusion_2}
Suppose $\mathrm{ran}(\Gamma_{XX \mid Y}) \subseteq \mathrm{ran}(\Sigma_{XX})$.
\end{assumption}

\begin{theorem}
Suppose Assumptions \ref{assu:L2_subset}, \ref{assu:L2_dense_1}, \ref{assu:L2_dense_3}, \ref{assu:inclusion_2} hold. Then $\Lambda_{\mathrm{SIR, D}}$ is a bounded operator, and $\overline{\mathrm{ran}} (\Lambda_{\mathrm{SIR, D}}) \subseteq \Scal_{Y \mid X}$. 
\end{theorem}

\section{Sample Estimation}
\label{sec:finite_sample}

In this section, we develop the sample estimator for the proposed metric SIR, first at the operator level, then under a coordinate system, given the i.i.d.\ random sample observations $\{ (X_1, Y_1), \ldots , (X_n, Y_n) \}$ of $(X, Y)$.

\subsection{Estimation at the operator level}

For the general case when the response $Y$ resides in a metric space, we first obtain the sample estimators of the mean elements by $\hat\mu_X = \E_n\{\kappa_X(\cdot, X)\}$, and $\hat\mu_Y = \E_n\{\kappa_Y(\cdot, Y)\}$, where $\E_n $ is the sample mean operator, such that $\E_n \omega = n^{-1} \sum_{i=1}^{n} \omega_i$ for the samples $\omega_1, \ldots, \omega_n$ from $\omega$. We next obtain the sample estimators of the covariance operators $\Sigma_{XX}, \Sigma_{XY}, \Sigma_{YY}$ in \eqref{eq:three_covariance_operators} as,
\begin{align*}
\begin{split}
\hat\Sigma_{XX} & = \E_n[ \{ \kappa_X(\cdot, X) - \hat\mu_X \} \otimes \{ \kappa_X(\cdot, X) - \hat\mu_X \} ], \\
\hat\Sigma_{XY} & = \E_n[ \{ \kappa_X(\cdot, X) - \hat\mu_X \} \otimes \{ \kappa_Y(\cdot, Y) - \hat\mu_Y \} ], \\
\hat\Sigma_{YY} & = \E_n[ \{ \kappa_Y(\cdot, Y) - \hat\mu_Y \} \otimes \{ \kappa_Y(\cdot, Y) - \hat\mu_Y \} ].
\end{split}
\end{align*}
Moreover, we have $\hat\Sigma_{YX} = \hat\Sigma_{XY}^*$. We then obtain the sample estimator of the metric SIR operator $\Lambda_{\mathrm{SIR}}$ in \eqref{eq:sir_operator_definition} as,
\begin{align*}
\hat\Lambda_{\mathrm{SIR}} = (\hat\Sigma_{XX} + \tau_{1} I)^{-1} \hat\Sigma_{XY} (\hat\Sigma_{YY} + \tau_{2} I)^{-1} \hat\Sigma_{YX},
\end{align*}
where we utilize the ridge regularization to estimate the pseudo-inverses, $\tau_1, \tau_2$ are the ridge parameters, and $I$ is the identity operator.  Finally, we estimate the range of $\Lambda_{\mathrm{SIR}}$ through the spectral decomposition of the operator $\hat\Lambda_{\mathrm{SIR}} \hat\Lambda_{\mathrm{SIR}}^*$. Suppose $\hat{f}_{1}, \ldots , \hat{f}_{d}$ are the $d$ leading eigenfunctions of $\hat\Lambda_{\mathrm{SIR}} \hat\Lambda_{\mathrm{SIR}}^*$. Then the estimated sufficient predictors corresponding to the observation $X  \in \Omega_X^0$ are $\hat{f}_{1}(X), \ldots , \hat{f}_{d}(X)$. 

For the special case when $Y$ resides in the usual Euclidean space and is discrete, we obtain the sample estimator of the covariance operator $\Gamma_{XX \mid Y}$ in \eqref{eq:discrete_operator} as,
\vspace{-0.05in}
\begin{align*}
\hat\Gamma_{XX|Y} = \frac{1}{n} \sum_{k = 1}^K n_k (\hat\gamma_{X \mid k} \otimes \hat\gamma_{X \mid k}),
\end{align*}
where $n_{k}$ is the number of samples belonging to the class $k$, $\mathbb{I}(\cdot)$ is the indicator function, and $\hat\gamma_{X \mid k} =  (n/n_k) \E_n  \{\mathbb{I}(Y = k) \kappa_X(\cdot, X) \} - \hat \mu_X$, for $k = 1, \ldots, K$. We then obtain the sample estimator of the metric SIR operator $\Lambda_{\mathrm{SIR, D}}$ in \eqref{eqn:msir-discrete} as, 
\begin{align*}
\hat\Lambda_{\mathrm{SIR, D}} = (\hat\Sigma_{XX} + \tau_{1} I)^{-1} \hat\Gamma_{XX \mid Y}.
\end{align*}
Finally, we estimate the range of $\Lambda_{\mathrm{SIR, D}}$ via the spectral decomposition of $\hat\Lambda_{\mathrm{SIR, D}} \hat\Lambda_{\mathrm{SIR, D}}^*$.

\subsection{Estimation under a coordinate representation}\label{sec:coordinate}

We next develop the estimation procedure under a chosen coordinate system. We divide the procedure into three main steps. We focus on the general case when $Y$ resides in a metric space, and briefly discuss the special case when $Y$ is discrete. 

In Step 1, we choose the kernel function $\kappa_X$ and $\kappa_Y$. There are multiple choices of kernel functions, while we employ the Gaussian kernel throughout our implementation. We use the leave-one-out cross-validation to determine the bandwidth parameters in $\kappa_X$ and $\kappa_Y$, following a similar strategy as in \cite{Lee:2013}. We then compute the Gram matrix $K_X = (\kappa_X(X_i, X_{i'}))_{i,i' = 1}^n \in \mathbb{R}^{n \times n}$, and $K_Y = (\kappa_Y(Y_i, Y_{i'}))_{i,i' = 1}^n \in \mathbb{R}^{n \times n}$, where the kernel functions $\kappa_X$ and $\kappa_Y$ are evaluated under the given metrics $d_X, d_Y$ as in \eqref{eq:kernel_metric}. Let $Q = I - n^{-1} 1 1\trans$ denote the centering matrix, where $1 \in \mathbb{R}^n$ is a vector of ones. We then compute the centered version of the Gram matrices as 
\begin{align} \label{eqn:GxGy}
G_X = Q K_X Q, \quad \textrm{ and } \quad G_Y = Q K_Y Q. 
\end{align}

In Step 2, we compute the coordinate representation of the sample metric SIR operator $\hat\Lambda_{\mathrm{SIR}}$. Toward that end, consider the sample counterpart of the space $\Hcal_X^0$, which is the span of the sample elements, $\hat\Hcal_{X}^0 = \mathrm{span} \big\{ \kappa_X( \cdot , X_i ) \mid i = 1, \ldots , n \big\}$. We impose the following linear independence assumption, which is a mild requirement. When it does not hold, we can simply delete a subset of the elements to obtain a linearly independent set. Alternatively, we can also construct a linearly independent basis via Karhunen-Lo\`eve expansion, see, e.g. \citet{LeeLi2022}.   

\begin{assumption}\label{assu:linearly_independent}
The elements $\kappa_X( \cdot , X_i )$, $i = 1, \ldots , n$, are linearly independent.  
\end{assumption}

\noindent
Under Assumption \ref{assu:linearly_independent}, the elements $\kappa_X( \cdot , X_i )$, $i = 1, \ldots , n$, form a basis for $\hat\Hcal_{X}^0$ and, given an arbitrary member $f \in \hat\Hcal_{X}^0$, we define its coordinate $[f] \in \mathbb{R}^n$ as the vector of its coefficients under this basis. As such, for any $f \in \hat\Hcal_{X}^0$ and $X \in \Omega_X^0$, $f(X) = [f]\trans k_X(X)$, where $k_X(X) = (\kappa_X(X , X_1), \ldots , \kappa_X(X , X_n))\trans$. In addition, we take the inner product of $\hat\Hcal_{X}^0$ to be the bilinear form, $(f, f') \mapsto \langle f, f' \rangle_{\hat\Hcal_{X}^0} = [f]\trans K_X [f']$, for $f, f' \in \hat\Hcal_{X}^0$, and the Gram matrix $K_X$ is ensured to be positive definite by Assumption~\ref{assu:linearly_independent}. Analogously, consider the sample counterpart of the space $\Hcal_X$, which is the span of the centered sample elements, $\hat\Hcal_{X} = \mathrm{span} \big\{ \kappa_X( \cdot, X_i) - \hat\mu_{X} \mid i = 1, \ldots , n \big\}$. We construct the sample spaces $\hat\Hcal_{Y}^0$ and $\hat\Hcal_{Y}$ similarly. 

Correspondingly, following \citet{Fukumizu:2009}, the coordinates of the sample covariance operators $\hat \Sigma_{XX}$, $\hat \Sigma_{XY}$, $\hat \Sigma_{YX}$, $\hat \Sigma_{YY}$ are, 
\vspace{-0.05in}
\begin{align*}
[\hat \Sigma_{XX}] = n^{-1}G_X, \quad [\hat \Sigma_{XY}] = n^{-1}G_Y, \quad [\hat \Sigma_{YX}] = n^{-1}G_X, \quad [\hat \Sigma_{YY}] = n^{-1}G_Y,
\end{align*}
where $G_X, G_Y$ are as defined in \eqref{eqn:GxGy}. We also clarify that, the above coordinate representation seems to suggest that $\hat{\Sigma}_{YX}$ does not depend on $Y$, which is not the case. Actually, $\hat \Sigma_{XX}$ and $\hat \Sigma_{YX}$ share the same coordinate, which is $n^{-1}G_X$, but they involve two different sets of bases, as $\hat \Sigma_{XX}$ and $\hat \Sigma_{YX}$ have different range spaces. For simplicity, we drop the involvement of the underlying bases in the coordinate bracket notation. But we remind that $\hat{\Sigma}_{YX}$ depends on $Y$ through the underlying bases. A similar discussion applies to  $\hat{\Sigma}_{XY}$ too.

We then obtain the coordinate representation of $\hat \Lambda_{\mathrm{SIR}}$ in the next lemma. Its proof follows immediately by the definition of $\hat \Lambda_{\mathrm{SIR}}$,  and is thus omitted.

\begin{lemma}\label{lem:sir_coordinate_1}
The metric SIR operator $\hat \Lambda_{\mathrm{SIR}}$ has the coordinate representation, 
\begin{align} \label{eqn:coordinate_mSIR}
[\hat \Lambda_{\mathrm{SIR}}] = G_X^\dagger G_Y G_Y^\dagger G_X,
\end{align}
where $^\dagger$ denotes the Moore-Penrose pseudo-inverse of a matrix.
\end{lemma}

To improve numerical stability, we replace the pseudo-inverse $G_X^\dagger$ in Lemma~\ref{lem:sir_coordinate_1} with its ridge-regularized counterpart $\{ G_X + \tau_1 I_n \}^{-1} $, where $\tau_1$ is taken to be $c \times \phi_1(G_X)$, $\phi_1(\cdot)$ is the largest eigenvalue of the designated matrix, and $c = 0.2$. A similar procedure was also employed in \citet{LeeLi2022}. Similarly, we replace  $G_Y^\dagger$ by $\{ G_Y + \tau_2 I_n \}^{-1} $ with $\tau_2 = c \times \phi_1(G_Y)$.
 
In Step 3, we estimate the range of $\hat \Lambda_{\mathrm{SIR}}$ through the eigen-decomposition of its coordinate in \eqref{eqn:coordinate_mSIR}. Letting $v_1, \ldots, v_d$ denote the $d$ leading eigenvectors of $[\hat \Lambda_{\mathrm{SIR}}] [\hat \Lambda_{\mathrm{SIR}}]\trans$, the estimated sufficient predictors corresponding to an observation $X \in \Omega_X^0$ are $v_1\trans Q k_X(X), \ldots v_d\trans Q k_X(X)$, where $k_X(X) = (\kappa_X(X , X_1), \ldots , \kappa_X(X , X_n))\trans$. Alternatively, one can also use the eigenvectors of the matrix $[\hat \Lambda_{\mathrm{SIR}}]$.

We remark that, the computational complexity of our proposed method is of the order $\mathcal{O}(n^3)$. When the sample size $n$ is huge, the computation can be intensive. For such a case, we propose an alternative estimation strategy similar to that of \cite{hung2019sufficient}. That is, we first divide all the sample observations into $Q$ disjoint subsets $\mathcal{I}_1, \ldots , \mathcal{I}_Q$. We then estimate the sufficient predictors given each subset $\mathcal{I}_q$, for $q = 1, \ldots, Q$. To accommodate for possible discrepancy in the signs of the resulting eigenvectors, we choose their signs such that, for each $j = 1, \ldots , d$, the sum $\sum_{q, q' = 1}^Q v_{j, q}\trans v_{j, q'} $ is maximized, where $v_{j, q}$ is the $j$th eigenvector of $[\hat \Lambda_{\mathrm{SIR}}] [\hat \Lambda_{\mathrm{SIR}}]\trans$ computed based on the $q$th subset $\mathcal{I}_q$. We then average the estimated sufficient predictors over all $Q$ subsets to produce the final estimate for the full samples.

For the special case when $Y$ resides in the usual Euclidean space and is discrete, the coordinate representation of $\gamma_{X \mid k}$ is $[\hat\gamma_{X \mid k}] = (1/n_k) 1_k - (1/n) 1$, where the $i$th element of the vector $1_k \in \mathbb{R}^n$ is the indicator $\mathbb{I}(Y_i = k)$, $i = 1, \ldots , n$. Correspondingly, the coordinate representation of $\hat \Lambda_{\mathrm{SIR,D}}$ is, 
\begin{align*}
[\hat \Lambda_{\mathrm{SIR, D}}] = G_X^\dagger Q \left(   \sum_{k=1}^K \frac{1}{n_k} 1_k 1_k\trans  \right) Q G_X.
\end{align*}

Finally, we briefly comment on the problem of selecting the reduced dimension $d$ in SDR. There have been a number of information criterion-based selection proposals for SDR of the Euclidean data \citep{zhu2006sliced, luo2009contour, xia2015consistently}. We expect a similar information criterion is applicable for our metric SIR as well, while we leave the full investigation as future research.

\section{Asymptotic Theory}
\label{sec:asymptotics}

In this section, we establish the convergence rate of the proposed metric SIR estimator at the operator level for both the general $Y$ and categorical $Y$ settings. 

We begin with some regularity conditions. 

\begin{assumption}\label{assu:continuous}
Suppose the kernel functions $\kappa_X$ and $\kappa_Y$ are continuous.
\end{assumption}

\begin{assumption}\label{assu:fourth_moments}
Suppose $\mathrm{E} \{ \kappa_X (X, X)^2 \} < \infty$, and $\mathrm{E} \{ \kappa_Y (Y, Y)^2 \} < \infty$.
\end{assumption}

\begin{assumption}\label{assu:asymptotics_regularity}
Suppose $\mathrm{ran}(\Sigma_{YX}) \subseteq \mathrm{ran}(\Sigma_{YY}^2)$, and $\mathrm{ran}(\Sigma_{XY}) \subseteq \mathrm{ran}(\Sigma_{XX}^2)$.
\end{assumption}

\noindent
Assumption \ref{assu:continuous} is quite mild, and together with the separability of the metric spaces $\Omega^0_X, \Omega^0_Y$, it ensures that the RKHS $\Hcal_X, \Hcal_Y$ are separable \citep{hein2004kernels}, which in turn ensures that $\Hcal_X, \Hcal_Y$ admit countable orthonormal bases. Assumption~\ref{assu:fourth_moments} is analogous to the requirement that a random variable has a finite fourth moment, and is reasonable.  Assumption~\ref{assu:asymptotics_regularity} can be seen as a stronger version of Assumption \ref{assu:inclusion}; that is, in comparison with Assumption \ref{assu:inclusion}, the mapping of $\Sigma_{XY}$ needs to concentrate even more on the leading eigen-spaces of $\Sigma_{XX}$ and $\Sigma_{YY}$. This, again, can be understood as a smoothness condition.

In our sample estimation, we employ the ridge regularization for the pseudo-inverses. For simplicity, in our theoretical analysis, we suppose the ridge parameters $\tau_1 = \tau_2=\tau$, and $\tau$ approaches zero as the sample size $n$ diverges. Denote the operator norm of a linear operator $A:\Hcal \to \Hcal'$ as $\|A\|_{\mathrm{OP}} = \sup\{\|Af\|_{\Hcal'}: \|f\|_{\Hcal} =1 \}$. The next theorem establishes the convergence of $\hat\Lambda_{\mathrm{SIR}}$ in terms of the operator norm for the general response case. 

\begin{theorem}\label{theo:asymptotics_theorem_1}
Suppose Assumptions \ref{assu:continuous} to \ref{assu:asymptotics_regularity} hold. Then, as $n \rightarrow \infty$, 
\begin{align*}
\left\| \hat\Lambda_{\mathrm{SIR}} -  \Lambda_{\mathrm{SIR}} \right\|_{\mathrm{OP}} = \mathcal{O}_p\left( \tau + \frac{1}{\tau \sqrt{n}} \right). 
\end{align*}
\end{theorem}

For the special case of $Y$ being categorical, we replace the smoothness condition of Assumption \ref{assu:asymptotics_regularity} with the following counterpart.  

\begin{assumption}\label{assu:asymptotics_regularity_2}
Suppose $\mathrm{ran}(\Gamma_{XX \mid Y}) \subseteq \mathrm{ran}(\Sigma_{XX}^2)$.
\end{assumption}

\begin{theorem}\label{theo:asymptotics_theorem_2}
Suppose Assumptions \ref{assu:continuous}, \ref{assu:fourth_moments}, and \ref{assu:asymptotics_regularity_2} hold. Then, as $n \rightarrow \infty$, 
\begin{align*}
\left\| \hat\Lambda_{\mathrm{SIR, D}}-  \Lambda_{\mathrm{SIR, D}} \right\|_{\mathrm{OP}} = \mathcal{O}_p\left( \tau + \frac{1}{\tau \sqrt{n}} \right). 
\end{align*}
\end{theorem}

\noindent
Theorems \ref{theo:asymptotics_theorem_1} and \ref{theo:asymptotics_theorem_2} suggest that our metric SIR estimator is consistent. Its convergence rate consists of two parts. The first part is due to the ridge regularization, and the second part represents the convergence of the sample operators to their population counterparts. If $\tau= n^{-\beta} $ for some constant $\beta > 0$, then the convergence rate becomes $n^{-\beta} + n^{\beta - 1/2} $, implying that the best possible convergence rate given by our result is $\mathcal{O}(n^{-1/4})$, achieved when $\beta = 1/4 $. We remark this is the same as the rate obtained by \citet{Li:2017} in nonlinear SDR for functional data.

\section{Numerical Studies}
\label{sec:examples}

In this section, we study the empirical performance of our proposed metric sliced inverse regression (MSIR), under difference choices of distance metrics. We also compare with the nonlinear SIR method of \citet[GSIR]{Lee:2013}. Although GSIR was originally formulated through the Euclidean geometry, it can be easily extended to incorporate an arbitrary distance metric.

\subsection{Torus manifold data}

As the first example, we consider a two-dimensional torus as the predictor, while we simulate the response using different distance metrics. A torus is best visualized as a unit square $[0, 1]^2$ for which the opposite edges have been ``glued together''. We consider two different generative models.
\begin{itemize}
\item[] Model 1: $Y_i = d_G\{X_i, (0.5, 0.5)\trans \} + \varepsilon_i$;
\item[] Model 2: $Y_i = d_G\{X_i, (1, 1)\trans \} + \varepsilon_i$,
\end{itemize}
where the 2-dimensional predictor $X_i$ is uniformly distributed in $[0, 1]^2$,  the error term $\varepsilon_i$ is drawn from a normal distribution with mean zero and variance $\sigma^2$, and $d_G$ denotes the geodesic distance. Since the point $(0.5, 0.5)\trans$ is in the middle of the unit square, we have $d_G\{X_i, (0.5, 0.5)\trans \} = d_E\{X_i, (0.5, 0.5)\trans \} $, where $d_E$ denotes the Euclidean distance. Consequently, in Model 1, the true relation between the response and the predictor is a smooth function of the Euclidean distance between the predictor and the center point of the square, and we expect the two distance functions to perform similarly under Model 1. The same is not true for Model 2, however, where the reference point $(1, 1)\trans$ lies at the corner of the square. This means that the true regression relationship is not a smooth function of the Euclidean distance, but it is so for the geodesic distance, making the geodesic distance more favorable under Model 2. For both models, we consider two sample sizes $n = 250, 500$, and two noise levels $\sigma = 0.05, 0.10$. We further divide the data into 80\% training samples, and 20\% testing samples. We consider two distance metrics, the geodesic distance and the Euclidean distance. 

\begin{table}[t!]
\centering
\caption{The torus data example: the average distance correlation (with the standard deviation in the parenthesis) between the response and estimated sufficient predictors.}
\vspace{1em}
\begin{tabular}{r|cc|cc} \hline
Model 1 & \multicolumn{2}{c|}{$n = 250$} & \multicolumn{2}{c}{$n = 500$} \\ \cline{2-5}
 & $ \sigma = 0.05 $ & $ \sigma = 0.10 $ & $ \sigma = 0.05 $ & $ \sigma = 0.10 $ \\ 
  \hline
MSIR $d_G$ & 0.912 (0.025) & 0.766 (0.058) & 0.911 (0.018) & 0.777 (0.038) \\ 
  GSIR $d_G$ & 0.719 (0.082) & 0.611 (0.088) & 0.715 (0.071) & 0.599 (0.081) \\
  \hline
  MSIR $d_E$ & 0.926 (0.021) & 0.779 (0.060) & 0.926 (0.014) & 0.790 (0.037) \\ 
  GSIR $d_E$ & 0.654 (0.092) & 0.563 (0.091) & 0.646 (0.083) & 0.552 (0.082) \\ 
   \hline
Model 2 & \multicolumn{2}{c|}{$n = 250$} & \multicolumn{2}{c}{$n = 500$} \\ \cline{2-5}
 & $ \sigma = 0.05 $ & $ \sigma = 0.10 $ & $ \sigma = 0.05 $ & $ \sigma = 0.10 $ \\ 
  \hline
MSIR $d_G$ & 0.912 (0.025) & 0.784 (0.054) & 0.913 (0.017) & 0.775 (0.040) \\ 
  GSIR $d_G$ & 0.726 (0.079) & 0.623 (0.094) & 0.724 (0.073) & 0.616 (0.082) \\
  \hline
  MSIR $d_E$ & 0.841 (0.046) & 0.729 (0.067) & 0.845 (0.032) & 0.722 (0.046) \\ 
  GSIR $d_E$ & 0.602 (0.087) & 0.526 (0.098) & 0.587 (0.084) & 0.509 (0.085) \\ 
   \hline
\end{tabular}
\label{tab:torus_1}
\end{table}

Table \ref{tab:torus_1} reports the distance correlation between the response and the first two estimated sufficient predictors evaluated on the testing samples and averaged over 200 data replications. It is seen that the proposed MSIR outperforms the competing GSIR, by achieving a higher distance correlation and a smaller standard error. Moreover, the Euclidean metric is slightly better suited to Model 1, where the toroidal geometry plays no role, while the geodesic metric is considerably better for Model 2, where the toroidal geometry plays a crucial role. The increased sample size mostly helps to reduce the standard error of the estimator. Figure \ref{fig:torus_1} further provides a visualization of the estimated sufficient predictors for a single data replication under Model 2 with $n=500$ and $\sigma=0.05$. It agrees with the qualitative patterns observed in Table \ref{tab:torus_1} that MSIR produces more informative sufficient predictors than GSIR.

\begin{figure}[t!]
\centering
\includegraphics[width=\textwidth, height=4.2in]{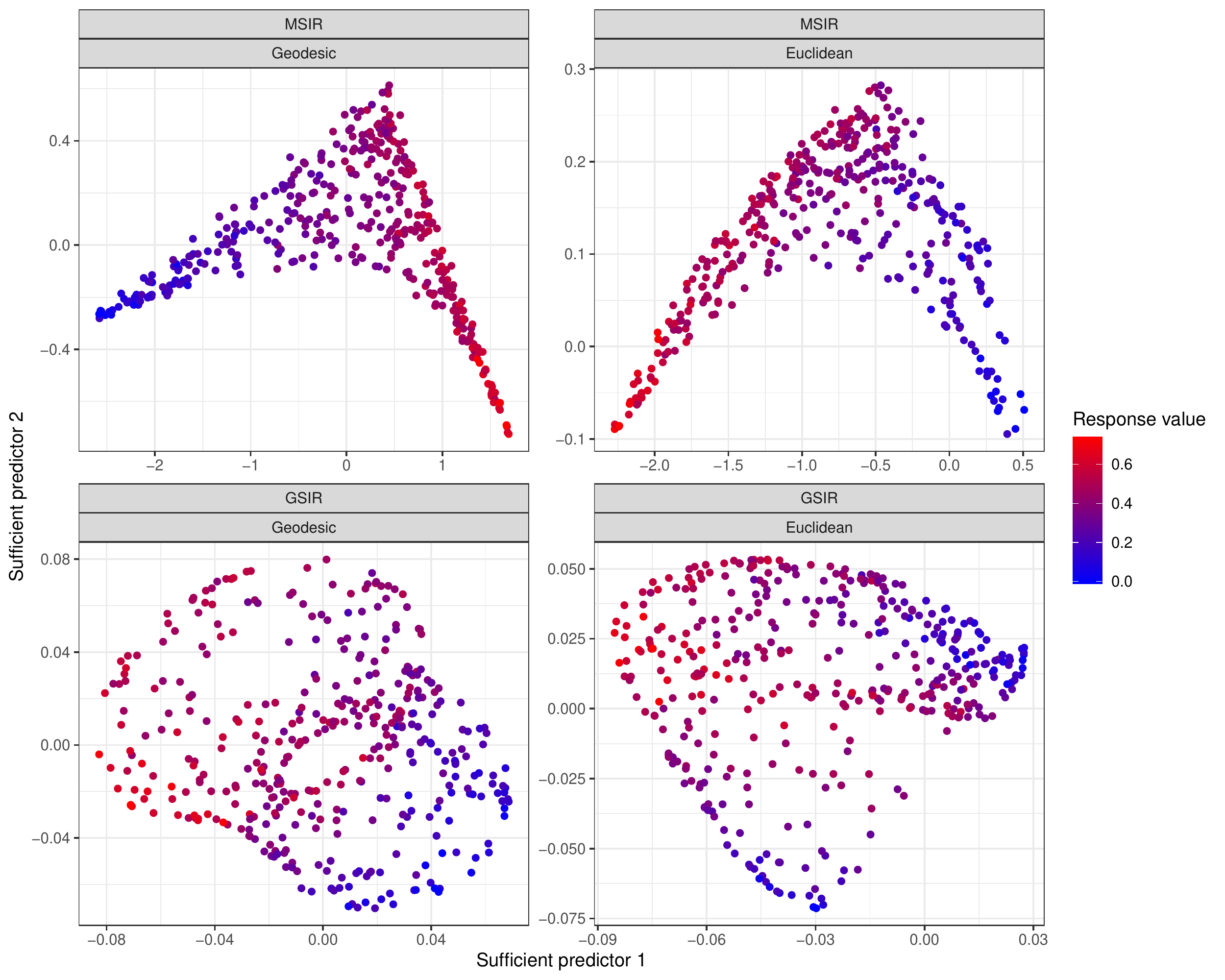}
\caption{The torus data example:  the sufficient predictors under two SDR methods and two distance metrics.}
\vspace{1em}
\label{fig:torus_1}
\end{figure}

\subsection{Positive definite matrix data}
\label{subsec:psd}

As the second example, we consider a positive definite matrix data example from a neuroimaging based autism study \citep{DiMartino2014}. Autism is an increasingly prevalent neurodevelopmental disorder, characterized by symptoms such as social difficulties, communication deficits, stereotyped behaviors and cognitive delays \citep{Rudie2013}. The dataset consists of $n=795$ subjects, among whom 362 were diagnosed with autism, and the rest healthy controls. For each subject, a resting-state functional magnetic resonance imaging (fMRI) scan was obtained, which measures the intrinsic functional architecture of the brain through the correlated synchronizations of brain systems. The corresponding brain functional connectivity network has been shown to alter under different disorders or during different brain developmental stages. Such alterations contain crucial insights of both disorder pathology and development of the brain \citep{Fox2010}. It is thus of great scientific importance to understand the association between the autism status and the brain connectivity network, and our goal is to produce sufficient predictors to separate the autism patients from those healthy controls. 

We follow the data processing procedure of \citet{sun2017store}, and summarize the brain connectivity network for each subject as a $116 \times 116$ correlation matrix, corresponding to the synchronizations of 116 brain regions-of-interest under the commonly used Anatomical Automatic Labeling atlas \citep{TzourioMazoyer2002}. Moreover, most of the observed connectivity matrices of this data are numerically rank-deficit, with the typical numerical rank ranging from 60 to 80. As such, we employ common principal components analysis, and project the connectivity matrices to the space of the top 30 common principal components, such that the minimal eigenvalue is at least $10^{-4}$ for each resulting matrix. 

We consider six distance metrics between two positive definite matrices $M_1$ and $M_2$. These include the affine invariant metric, $d_A(M_1, M_2) = \| \mathrm{Log} (M_1^{-1/2} M_2 M_1^{-1/2}) \|_{F}$, where $\mathrm{Log}(\cdot)$ denotes the matrix logarithm, and $\| \cdot \|_{F}$ the Frobenius norm, the log-Euclidean metric, $d_{LE}(M_1, M_2) = \| \mathrm{Log}(M_1) - \mathrm{Log}(M_2) \|_{F}$, the S-divergence \citep{sra2016positive}, $d_S(M_1, M_2) = \log | (M_1 + M_2)/2 | - (1/2) \log | M_1 M_2 | $, where $| \cdot |$ denotes the determinant, the symmetrized Kullback-Leibler divergence, $d_{KL}(M_1, M_2) = \{ h(M_1, M_2) + h(M_2, M_1) \} / 2$, where $h(M_1, M_2) = \{ \mathrm{tr} (M_1^{-1} M_2) + \log|M_1| - \log|M_2| \} / 2$, the standard Euclidean metric, $d_E(M_1, M_2) = \| M_1 - M_2 \|_{F}$, and the Pearson metric, $d_P(M_1, M_2) = \| M_1/\|M_1\|_{F} - M_2/\|M_2\|_{F} \|_{F}$. Among these six distance metrics, the first three properly acknowledge the geometry of the matrix space $\mathcal{M}_d$, the fourth one hinges on the normality distribution, and the last two only leverage the Euclidean geometry.

\begin{figure}[t!]
\centering
\includegraphics[width=0.9\textwidth]{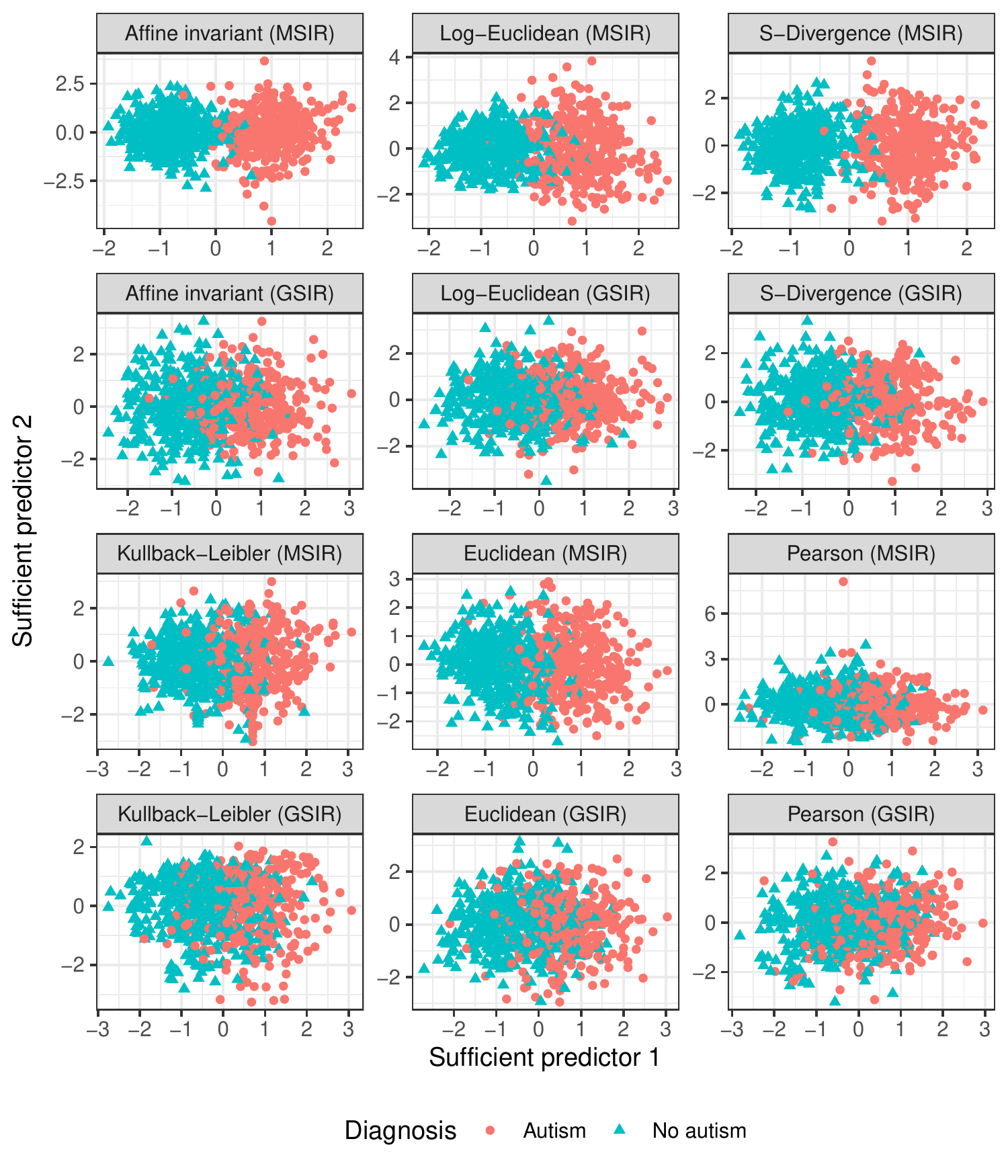}
\caption{The positive definite matrix data example: the sufficient predictors under two SDR methods and six metrics, with two groups of subjects, autism or control, marked by different colors.}
\vspace{1em}
\label{fig:psd_1}
\end{figure}

\begin{table}[t!]
\centering
\caption{The positive definite matrix data example: the leave-one-out cross-validation prediction error under two SDR methods and three metrics.}
\label{tab:psd_1}
\vspace{1em}
\begin{tabular}{r|ccc}\hline
 & Affine invariant & S-divergence & Euclidean \\
  \hline
MSIR & 0.306 & 0.302 & 0.333  \\
GSIR & 0.319 & 0.328 & 0.357  \\
   \hline
\end{tabular}
\end{table}

Figure \ref{fig:psd_1} shows the first two estimated sufficient predictors graphically. It is seen that the first sufficient predictors found by MSIR and GSIR are both able to separate the two groups of subjects to a good extent, whereas MSIR achieves generally a better separation than GSIR. Moreover, the first three distance metrics achieve a better separation than the last three metrics, which agrees with our expectation. Table \ref{tab:psd_1} reports the leave-one-out cross-validation prediction error when applying a quadratic discriminant analysis classifier to the extracted first two sufficient predictors. For simplicity, we only consider three metrics,  the affine invariant metric, and the S-divergence metric, due to their competitive performance as shown in Figure~\ref{fig:psd_1}, and the Euclidean metric, which serves as a benchmark. It confirms with the visual observation from Figure \ref{fig:psd_1} that MSIR outperforms GSIR, and the metrics that acknowledge the matrix geometry outperform the one that does not.

\subsection{Compositional data}

As the final example, we consider a compositional dataset from a gut microbiota study \citep{guo2016intestinal}. The dataset consists of $n = 83$ subjects, among whom 41 suffered from gout, and the rest not. For each subject, $p = 3684$ operational taxonomic units (OTUs) were measured, which characterizes the structure of the subject's intestinal microbiota. It is of great scientific interest to understand the association between the gout status and the OTU compositions \citep{guo2016intestinal}, and we aim to produce sufficient predictors to reflect the gout status. 

We follow the data processing procedure of \cite{pan2020ball} who analyzed the same data. Specifically, we first standardize the OTUs, so that the OTU measurements for each subject sum to one, and thus the data are compositional. In addition, the data are highly sparse, in that, on average, only 202 out of 3684 measurements are non-zero. As in \cite{pan2020ball}, we map the standardized vector to the $p$-dimensional unit sphere by taking element-wise square roots of the coordinates. 

We consider three distance metrics. The first metric is the arc length distance between two transformed compositions. The second metric is the Hamming distance evaluated on the dichotomized transformation of the compositions; i.e., all the nonzero entries are turned into one. This is motivated by the observation that the compositions are very sparse, and the positions rather than the magnitudes of the nonzero entries are more relevant. The third metric is the usual Euclidean distance. 

\begin{figure}[t!]
\centering
\includegraphics[width=1\textwidth]{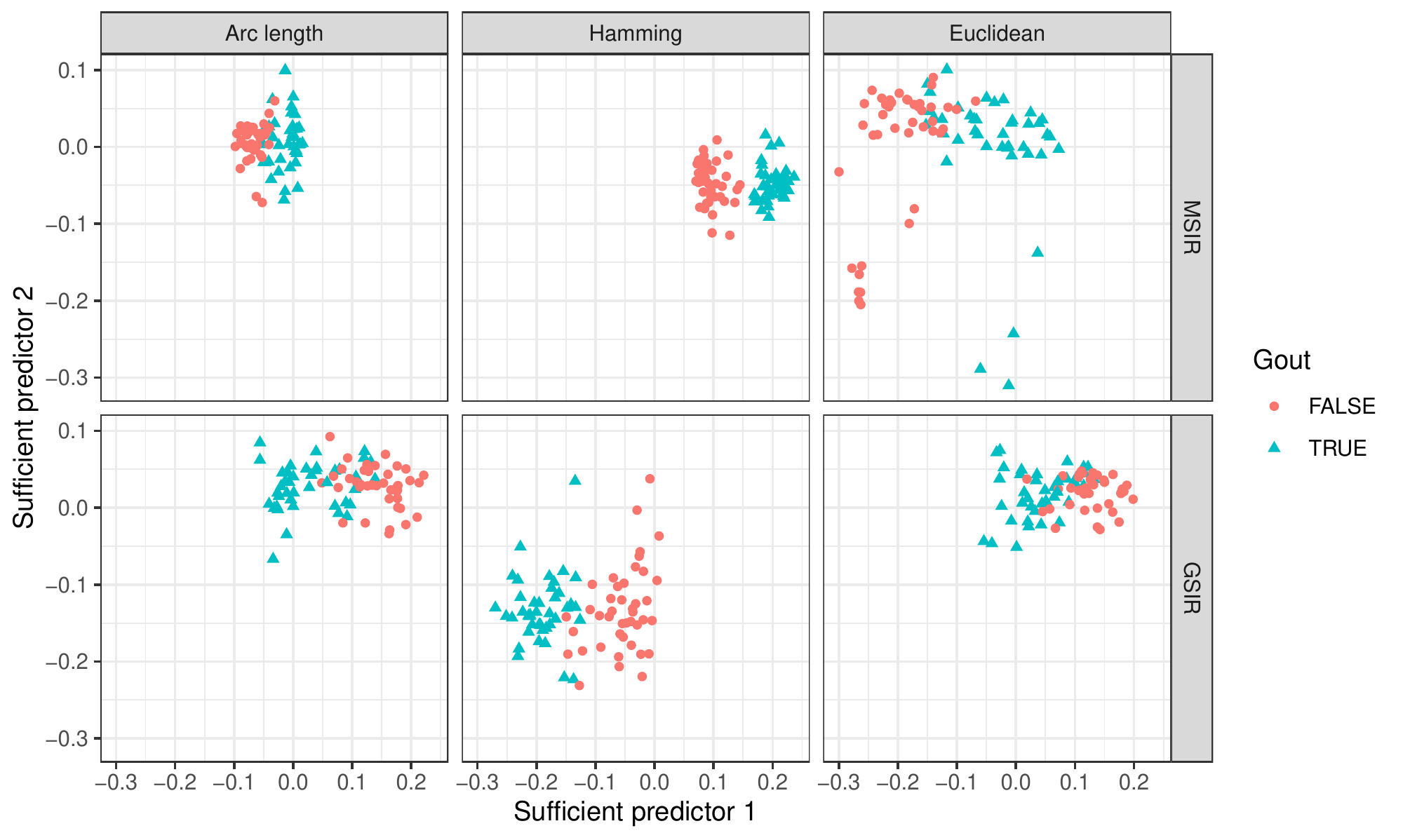}
\caption{The compositional data example: the sufficient predictors under two SDR methods (row) and three metrics (column), with two groups of subjects, gout or not, marked by different colors.}
\vspace{1em}
\label{fig:simplex_1}
\end{figure}

Figure~\ref{fig:simplex_1} shows the estimated top two sufficient predictors graphically. It is seen that the first sufficient predictors found by MSIR and GSIR are both able to separate the two groups of subjects to some extent. Particularly, MSIR with the Hamming distance metric achieves the best separation. Table \ref{tab:simplex_1} reports the leave-one-out cross-validation prediction error when applying a quadratic discriminant analysis classifier to the extracted sufficient predictors when $d$ is taken as 1 and 2, respectively. Again, the proposed MSIR with the Hamming distance metric achieves the best prediction accuracy. Moreover, there is little difference between $d=1$ and $d=2$, suggesting a single summary predictor is sufficient, which agrees with our expectation since the response is only binary. 

\begin{table}[t!]
\centering
\caption{The compositional data example: the leave-one-out cross-validation prediction error under two SDR methods, three metrics, and two working dimensions.}\label{tab:simplex_1}
\vspace{1em}
\begin{tabular}{ll|rrr} \hline
$d$ & Method & Arc length & Hamming & Euclidean \\ 
  \hline
\multirow{2}{*}{1} & MSIR & 0.241 & 0.229 & 0.253 \\ 
& GSIR & 0.253 & 0.229 & 0.289 \\ 
   \hline
\multirow{2}{*}{2} & MSIR & 0.229 & 0.229 & 0.277 \\ 
& GSIR & 0.253 & 0.229 & 0.289 \\ \hline
\end{tabular}
\end{table}

To conclude this study, we give an example on how to interpret the obtained sufficient predictors. The key idea is to compute the correlations between the sufficient and original predictors. Figure \ref{fig:simplex_2} shows the histograms of the correlations between the first sufficient predictor obtained by MSIR and the original predictor under the three metrics, which demonstrate a relatively clear bimodal pattern. By Figure \ref{fig:simplex_1}, a large value of the first MSIR sufficient predictor indicates the presence of gout in a subject. As such, we expect the rightmost peaks of the three histograms in Figure~\ref{fig:simplex_2} to correspond to OTUs that are associated with gout. To confirm this, we note that \cite{guo2016intestinal} identified the OTUs of the geni \textit{Coprococcus} (78 in total) and \textit{Barnesiella} (14 in total) as the ones mostly associated with non-gouty and gouty subjects, respectively. The OTUs of these two geni have been colored in the rugs below the histograms of Figure~\ref{fig:simplex_2} and they are indeed roughly divided between the two modes of the histograms, with \textit{Coprococcus} concentrating to the left peak and \textit{Barnesiella} to the right. This effect is most pronounced in the middle histogram corresponding to the Hamming distance, which is in line with our result that the Hamming distance gives the best performance out of the three distance metrics. 

\begin{figure}[t!]
\centering
\includegraphics[width=1\textwidth, height=2.25in]{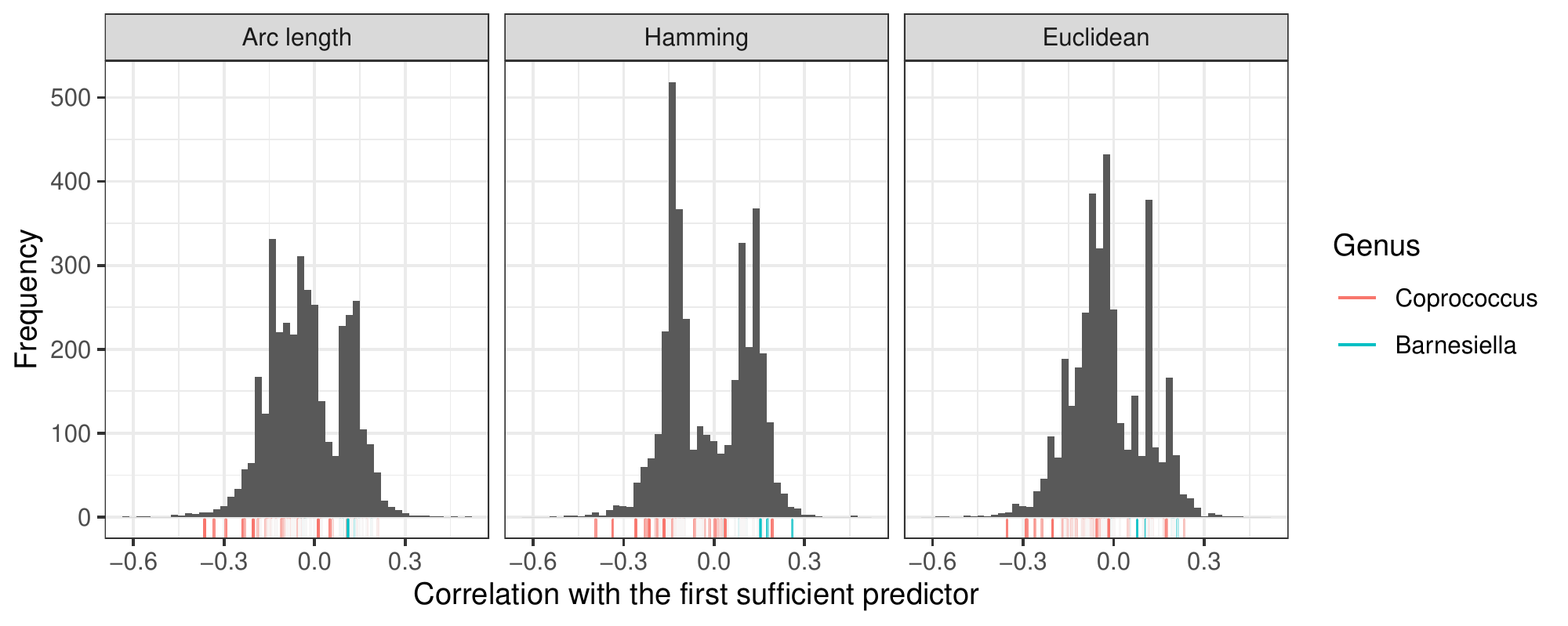}
\caption{The compositional data example: the histograms of the correlations between the first sufficient predictor obtained by MSIR and the original predictor under the three metrics.}
\vspace{1em}
\label{fig:simplex_2}
\end{figure}



\subsection*{Acknowledgements}
The authors thank the Editor, the Associate Editor, and the anonymous referee for their constructive comments. Virta's research is supported by the Academy of Finland (Grant 335077). Lee's research is partially supported by the NSF grant CIF-2102243, and the Seed Funding grant from Fox School of Business, Temple University. Li's research is partially supported by the NSF grant CIF-2102227 and the NIH grant R01AG061303.

\appendix

\section{Proofs of the theoretical results}\label{sec:proofs}

\vspace{-0.1in}
\noindent
\textbf{Proof of Lemma \ref{lem:almost_sure_representer}}:
As the metric space $(\Omega_X^0, d_X)$ is separable, so is $\Hcal_X^0$ \citep[Lemma 4.3]{lukic2001stochastic}. Being a closed subspace of a separable Hilbert space,  $\mathrm{ker}(\Sigma_{XX}^0)$ is also a separable Hilbert space, and hence admits a countable dense subset $\mathcal{K}$. Fixing $h \in \mathcal{K}$, we have $ \mathrm{Var}\{ h(X) \}=0$, implying that there exists a set $\mathcal{S}_{Xh}$, such that $P_X(\mathcal{S}_{Xh}) = 1$ and that, for all $x \in \mathcal{S}_{Xh}$, we have $ h(x) - \mathrm{E}\{ h(X) \} = \langle h, \kappa_X(\cdot, x) - \mu_X \rangle_{\Hcal_X^0}=0$.

Denoting $\Omega_X = \cap_{h \in \mathcal{K}} \mathcal{S}_{Xh}$, the countability of $\mathcal{K}$ implies that $P_X(\Omega_{X}) = 1$. We next show that, for all $x \in \Omega_X$, we have $\kappa_X( \cdot , x) - \mu_X \in \mathrm{ker}(\Sigma_{XX}^0)^\bot = \Hcal_X$. Taking an arbitrary $g \in \mathrm{ker}(\Sigma_{XX}^0)$, there exist a sequence of elements $h_j \in \mathcal{K}$, such that $\| g - h_j  \|_{\Hcal_X^0} \rightarrow 0 $ as $j \rightarrow \infty$. Let $\mathbb{N}$ denote the collection of natural numbers. Then, for an arbitrary $x \in \Omega_X$ and $j \in \mathbb{N}$, we have that,
\begin{align*}
	|\langle \kappa_X(\cdot, x) - \mu_X, g \rangle_{\Hcal_X^0} | &\leq |\langle \kappa_X(\cdot, x) - \mu_X, g - h_j \rangle_{\Hcal_X^0} | + |\langle \kappa_X(\cdot, x) - \mu_X, h_j \rangle_{\Hcal_X^0} | \\
	&\leq \| \kappa_X(\cdot, x) - \mu_X \|_{\Hcal_X^0} \| g - h_j \|_{\Hcal_X^0} + 0,
\end{align*}
which further implies that $\langle \kappa_X(\cdot, x) - \mu_X, g \rangle_{\Hcal_X^0} = 0$. This completes the proof of Lemma \ref{lem:almost_sure_representer}.
\eop

\bigskip
\medskip
\noindent
\textbf{Proof of Lemma \ref{lem:sir_operator}}:
\textit{(a)}. The proof is structurally similar to that of Proposition 1 in \cite{Li:2017}, but we include it for completeness.

Fix $f \in \Hcal_X$. Denote $\Sigma_{YY}^\dagger \Sigma_{YX} =R_{YX}$, and fix an arbitrary $g \in \Hcal_Y$. Then,
\begin{align*}
	\mathrm{Cov}\{ (R_{YX} f )(Y), g(Y) \} = \langle \Sigma_{YY}^\dagger \Sigma_{YX} f, \Sigma_{YY} g \rangle_{\Hcal_Y}&= \langle \Sigma_{YX} f, g \rangle_{\Hcal_Y} = \mathrm{Cov}\{ f(X), g(Y) \}.
\end{align*}
Consequently, for all $f \in \Hcal_X$ and $g \in \Hcal_Y$, we have, 
\begin{align}\label{eq:conditional_expectation_l2}
	\mathrm{Cov}\{ f(X) - (R_{YX} f)(Y), g(Y) \} = 0,
\end{align}
Consider an arbitrary $h \in L_2(P_Y)$. By Assumption \ref{assu:L2_dense_1}, there exist a sequence $\{ h_n \}$ of elements of $ \Hcal_Y$, such that $ \mathrm{var}\{h(Y) - h_n(Y)\} \rightarrow 0$, as $n \rightarrow \infty$. Therefore, by \eqref{eq:conditional_expectation_l2},
\begin{align*}
	& | \mathrm{Cov}\{ f(X) - (R_{YX} f)(Y), h(Y) \} |\\
	\leq \; & | \mathrm{Cov}\{ f(X) - (R_{YX} f)(Y), h(Y) - h_n(Y) \} | + | \mathrm{Cov}\{ f(X) - (R_{YX} f)(Y), h_n(Y) \} | \\
	\leq \; & [ \mathrm{Var}\{ f(X) - (R_{YX} f)(Y) \} \,\mathrm{Var} \{ h(Y) - h_n(Y) \} ]^{1/2} + 0, 
\end{align*}
for all $n$. The first variance in the final expression above is finite, since $\mathrm{Var}\{ f(X) \} \leq \| \Sigma_{XX} \|_{\mathrm{OP}} \| f \|^2_{\Hcal_X} < \infty$, and $\mathrm{Var}\{  (R_{YX} f)(Y) \} \leq \| R_{YX} \|_{\mathrm{OP}} \| \Sigma_{YX} \|_{\mathrm{OP}} \| f \|^2_{\Hcal_X} < \infty$. This implies that \eqref{eq:conditional_expectation_l2} holds also when $g$ is replaced with any $h \in L_2(P_Y)$. 

Note that a square-integrable random variable $Z$ is almost surely equal to the conditional expectation $\mathrm{E}\{ f(X) \mid Y \}$, if 
\begin{align}\label{eq:conditional_expectation_l2_2}
	\mathrm{E}[\{ f(X) - Z \} h (Y) ] = 0,
\end{align}
for all $h \in L_2(P_Y)$. A direct computation using \eqref{eq:conditional_expectation_l2} shows that the choice $Z = (R_{YX} f)(Y) - \mathrm{E} \{ (R_{YX} f)(Y) \} + \mathrm{E} \{ f(X) \} $ satisfies \eqref{eq:conditional_expectation_l2_2} for all $h \in L_2(P_Y)$, which implies the following holds almost surely, 
\begin{align}\label{eq:almost_sure_first_conditional}
	\mathrm{E}\{ f(X) \mid Y \} - \mathrm{E} \{ f(X) \} = (\Sigma_{YY}^\dagger \Sigma_{YX} f)(Y) - \mathrm{E} \{ (\Sigma_{YY}^\dagger \Sigma_{YX} f)(Y) \}.
\end{align}

Finally, by Lemma \ref{lem:almost_sure_representer}, the right-hand side of \eqref{eq:almost_sure_first_conditional} is almost surely equal to the random variable $\langle \Sigma_{YY}^\dagger \Sigma_{YX} f, \kappa_Y( \cdot, Y ) - \mu_Y \rangle_{\Hcal_Y}$, where we take $\kappa_Y( \cdot, Y ) - \mu_Y$ to equal the zero element for those values of $Y$ for which it is not a member of $\Hcal_Y$. This proves the assertion (a).

\medskip
\textit{(b)}. By definition,
\begin{align*}
	\mathrm{E} \langle  g, [ \{ \kappa_Y( \cdot, Y ) - \mu_Y \} \otimes \{ \kappa_Y( \cdot, Y ) - \mu_Y \} ] g' \rangle_{\Hcal_Y} =  \langle  g, \Sigma_{YY} g' \rangle_{\Hcal_Y},
\end{align*}
for all $g, g' \in \Hcal_Y$. This, in conjunction with part (a) of the lemma, implies that the left-hand side of the assertion (b) equals,
\begin{align*}
	& \mathrm{E} \langle \Sigma_{YY}^\dagger \Sigma_{YX} f, [ \{ \kappa_Y( \cdot, Y ) - \mu_Y \} \otimes \{ \kappa_Y( \cdot, Y ) - \mu_Y \} ] \Sigma_{YY}^\dagger \Sigma_{YX} f' \rangle_{\Hcal_Y} \\
	= \; & \langle \Sigma_{YY}^\dagger \Sigma_{YX} f, \Sigma_{YY} \Sigma_{YY}^\dagger \Sigma_{YX} f' \rangle_{\Hcal_Y}
	= \; \langle f, \Sigma_{XY} \Sigma_{YY}^\dagger \Sigma_{YX} f' \rangle_{\Hcal_X}.
\end{align*}
This proves the assertion (b), and completes the proof of Lemma \ref{lem:sir_operator}.
\eop

\bigskip
\medskip
\noindent
\textbf{Proof of Theorem \ref{theo:SIR_range_closure}}:
By Theorem 1 in \cite{douglas1966majorization}, and Assumption \ref{assu:inclusion}, both $\Sigma_{YY}^\dagger \Sigma_{YX}$ and $\Sigma_{XX}^\dagger \Sigma_{XY}$ are bounded. Henceforth, $\Lambda_{\mathrm{SIR}}$ is also bounded.

To prove the unbiasedness of $\Lambda_{\mathrm{SIR}}$, we first note that the set $\Scal_{Y \mid X}$ is closed because measurability is preserved in taking point-wise limits, which in an RKHS is implied by the convergence in norm. Concurrently, $\Scal_{Y \mid X} = \overline{\mathrm{span}} \{ f \in \Hcal_X \mid f \mbox{ is } \Gcal_{Y \mid X}\mbox{-measurable } \}$, and we have the desired result of $\overline{\mathrm{ran}} ( \Lambda_{\mathrm{SIR}} ) \subseteq  \Scal_{Y \mid X}$, as long as we can show that $ \Scal_{Y \mid X}^\bot \subseteq \mathrm{ker} ( \Lambda^*_{\mathrm{SIR}} )$.

We begin by establishing this inclusion for the elements of $\mathrm{ran}(\Sigma_{XX})$. Let $f =\Sigma_{XX} m$ for an arbitrary $ m \in \Hcal_X$. Suppose $\langle f, h \rangle_{\Hcal_X} = 0 $ for all $h \in \Scal_{Y \mid X}$, which implies that, $ \mathrm{Cov}\{ m(X), h(X) \} = \langle \Sigma_{XX} m, h \rangle_{\Hcal_X} = 0$ for all $h \in \Scal_{Y \mid X}$. Let $\Scal^*_{Y \mid X} = \{ h \in L_2(P_X) \mid h \mbox{ is } \mathcal{G}_{Y \mid X} \mbox{-measurable} \}$. Then, by \cite[Theorem 13.3]{Libook}, we have that $\mathrm{Cov}\{ m(X), h(X) \} = 0$ for all $h \in \Scal^*_{Y \mid X}$. This in turn implies that 
\begin{align*}
	\langle m - \mathrm{E}\{ m(X) \}, h  \rangle_{L_2(P_X)} = 0,
\end{align*}
for all $h \in \Scal^*_{Y \mid X}$, where $\mathrm{E}\{ m(X) \}$ represents the constant function taking the value $\mathrm{E}\{ m(X) \}$ everywhere. Therefore, following \citet[Lemma 1]{Lee:2013}, we have that
\begin{align}\label{eq:gyx_conditional_expectation}
	\mathrm{E} \{ m(X) \mid \mathcal{G}_{Y \mid X} \} - \mathrm{E}\{ m(X) \} = 0, \quad \textrm{almost surely}.
\end{align}

We next show that \eqref{eq:gyx_conditional_expectation} leads to $\mathrm{E} \{ m(X) \mid Y \} = \mathrm{E}\{ m(X) \}$ almost surely. Let $\sigma(Y, \mathcal{G}_{Y \mid X})$ be the smallest $\sigma$-field containing both $\sigma(Y)$ and $\mathcal{G}_{Y \mid X}$. By rule of iterative expectation, we have, almost surely, 
\begin{align*}
	\mathrm{E} \{ m(X) \mid Y \} = \mathrm{E} [ \mathrm{E} \{ m(X) \mid \sigma(Y, \mathcal{G}_{Y \mid X}) \} \mid Y ] = \mathrm{E} [ \mathrm{E} \{ m(X) \mid \mathcal{G}_{Y \mid X} \} \mid Y ] = \mathrm{E}\{ m(X) \},
\end{align*}
where the second equality follows from the fact that $Y \indep X \mid \mathcal{G}_{Y \mid X}$, and the third equality is by \eqref{eq:gyx_conditional_expectation}.     

Combining the above result with Lemma \ref{lem:sir_operator} leads to that, for all $g \in \Hcal_X$, 
\begin{align*}
	0 = \langle m, \Sigma_{XY} \Sigma_{YY}^{\dagger} \Sigma_{YX} g \rangle_{\Hcal_X}
	= \langle f, \Lambda_{\mathrm{SIR}} g \rangle_{\Hcal_X}.
\end{align*}
In other words, $f \in \mathrm{ker} ( \Lambda_{\mathrm{SIR}}^* )$. Therefore, $ \mathrm{ran}(\Sigma_{XX}) \cap \Scal_{Y \mid X}^\bot \subseteq \mathrm{ker} ( \Lambda_{\mathrm{SIR}}^* )$.

To extend this inclusion to hold in the full orthogonal complement $ \Scal_{Y \mid X}^\bot $, we invoke Assumption \ref{assu:L2_dense_3}, which implies that, for $f \in \Scal_{Y \mid X}^\bot $, there exist a sequence of elements $f_n$ of $ \mathrm{ran}(\Sigma_{XX}) \cap \Scal_{Y \mid X}^\bot$, such that $\| f_n - f \|_{\Hcal_X} \rightarrow 0$, as $n \rightarrow 0$. Because $\Lambda_{\mathrm{SIR}}^* f_n = 0$ for all $n$, we also have $\Lambda_{\mathrm{SIR}}^* f = 0$ by continuity. This completes the proof of Theorem \ref{theo:SIR_range_closure}.
\eop

\bigskip
\medskip
\noindent
\textbf{Proof of Theorem \ref{theo:asymptotics_theorem_1}}:
We first present three auxiliary lemmas, under the same set of conditions of Theorem \ref{theo:asymptotics_theorem_1}. We then prove Theorem \ref{theo:asymptotics_theorem_1} based on these lemmas.

The first auxiliary lemma shows that the sample covariance operators are root-$n$ consistent estimators of the corresponding population counterparts.

\begin{lemma}\label{lem:technical_1}
	Suppose the conditions of Theorem \ref{theo:asymptotics_theorem_1} hold. Then, $\| \hat{\Sigma}_{XX} - \Sigma_{XX} \|_{\mathrm{HS}}$, $\| \hat{\Sigma}_{XY} - \Sigma_{XY} \|_{\mathrm{HS}}$ and $\| \hat{\Sigma}_{YY} - \Sigma_{YY} \|_{\mathrm{HS}}$ are of the order $\mathcal{O}_p(1/\sqrt{n})$.
\end{lemma}

\begin{proof}[Proof of Lemma \ref{lem:technical_1}]
	Denote $h_X = \kappa_X( \cdot, X ) - \mu_X$. By definition, the covariance operator $\Sigma_{XX}$ satisfies that,
	\begin{align*}
		\langle f, \Sigma_{XX} g \rangle_{\Hcal_X} = \mathrm{E} ( \langle f, h_X \rangle_{\Hcal_X} \langle g, h_X \rangle_{\Hcal_X} ) = \mathrm{E} \{ \langle f, (h_X \otimes h_X) g \rangle_{\Hcal_X} \} ,
	\end{align*}
	where $h_X$ is to be the zero element for those realizations $x \in \Omega$ not belonging to the almost sure set in Lemma \ref{lem:almost_sure_representer}.
	
	Since the covariance operator in a separable Hilbert space is a trace-class operator \citep{Blanchard1999}, we have that $\| \Sigma_{XX} \|_{\mathrm{HS}} < \infty $. To show $\| \hat{\Sigma}_{XX} - \Sigma_{XX} \|_{\mathrm{HS}} = \mathcal{O}_p(1/\sqrt{n})$, we note that,
	\begin{align*}
		\hat{\Sigma}_{XX} = \frac{1}{n} \sum_{i=1}^n (b_{X_i} \otimes b_{X_i}), 
	\end{align*}
	where $b_{X_i} = \kappa_X(\cdot, X_i) - (1/n) \sum_{j=1}^n \kappa_X(\cdot, X_j)$. Denoting $h_{X_i} = \kappa_X(\cdot, X_i) - \mu_X$, we further have that $b_{X_i} = h_{X_i} - \bar{h}_n $ where $ \bar{h}_n = (1/n) \sum_{i=1}^n h_{X_i} $. Therefore,
	\begin{align}\label{eq:covariance_convergence_parts}
		\| \hat{\Sigma}_{XX} - \Sigma_{XX} \|_{\mathrm{HS}} \leq \left\| \frac{1}{n} \sum_{i=1}^n (h_{X_i} \otimes h_{X_i}) - \Sigma_{XX} \right\|_{\mathrm{HS}} + \| \bar{h}_n \otimes \bar{h}_n \|_{\mathrm{HS}}.
	\end{align}
	
	For the second term on the right-hand-side of \eqref{eq:covariance_convergence_parts}, we have that,
	\begin{align}\label{eq:covariance_convergence_part_1}
		\mathrm{E} \| \bar{h}_n \otimes \bar{h}_n \|_{\mathrm{HS}} = \mathrm{E} \| \bar{h}_n \|^2_{\Hcal_X} = \frac{1}{n^2} \sum_{i=1}^n \sum_{j=1}^n \mathrm{E} \langle h_{X_i}, h_{X_j} \rangle_{\Hcal_X}.
	\end{align}
	For any pair of distinct indices $i \neq j$, $ h_{X_i} $ and $ h_{X_j} $ are independent. This means that, 
	\begin{align}\label{eq:parseval_argument}
		\begin{split}
			\mathrm{E} \langle h_{X_i}, h_{X_j} \rangle_{\Hcal_X} = \mathrm{E}_{X_j} \langle \mathrm{E}_{X_i} (h_{X_i}), h_{X_j} \rangle_{\Hcal_X} = \mathrm{E}_{X_j} \langle 0, h_{X_j} \rangle_{\Hcal_X} = 0,
		\end{split}
	\end{align}
	where $\mathrm{E}_{X_i}(\cdot)$ is the expectation with respect to the distribution of $X_i$. Consequently, 
	\begin{align*}
		\mathrm{E} \| \bar{h}_n \otimes \bar{h}_n \|_{\mathrm{HS}} = \frac{1}{n^2} \sum_{i=1}^n \mathrm{E} \| h_{X_i}\|^2_{\Hcal_X} = \frac{1}{n} \mathrm{E} \| h_{X_1}\|^2_{\Hcal_X}, 
	\end{align*}
	where $\mathrm{E} \| h_{X_1} \|_{\Hcal_X}^2 = \mathrm{E} \| h_{X_1} \|_{\Hcal_X^0}^2 = \mathrm{E} \{ \kappa_X(X_1, X_1) \} - \| \mu_X \|_{\Hcal_X^0}^2 < \infty$, which holds by Assumption~\ref{assu:fourth_moments}. Therefore, $\mathrm{E} \| \bar{h}_n \otimes \bar{h}_n \|_{\mathrm{HS}} = \mathcal{O}(1/n)$, which, by Markov's inequality, further implies that $\| \bar{h}_n \otimes \bar{h}_n \|_{\mathrm{HS}} = \mathcal{O}_p(1/n)$. 
	
	For the first term on the right-hand-side of \eqref{eq:covariance_convergence_parts}, we have that, 
	\begin{align}\label{eq:covariance_convergence_part_2}
		\mathrm{E} \left\| \frac{1}{n} \sum_{i=1}^n (h_{X_i} \otimes h_{X_i}) - \Sigma_{XX} \right\|_{\mathrm{HS}}^2 = \frac{1}{n^2} \sum_{i=1}^n \sum_{j=1}^n \mathrm{E} \langle H_i, H_j  \rangle_{\mathrm{HS}},
	\end{align}
	where $H_i = (h_{X_i} \otimes h_{X_i}) - \Sigma_{XX}$. By the definition of the Hilbert-Schmidt norm,
	\begin{align*}
		\mathrm{E} \langle H_i, H_j  \rangle_{\mathrm{HS}} = \mathrm{E} \langle h_{X_i} \otimes h_{X_i}, h_{X_j} \otimes h_{X_j}  \rangle_{\mathrm{HS}} - \langle \Sigma_{XX}, \Sigma_{XX}  \rangle_{\mathrm{HS}}.
	\end{align*}
	Therefore, adopting he same strategy used to simplify \eqref{eq:covariance_convergence_part_1}, we have $\mathrm{E} \langle h_{X_i} \otimes h_{X_i}, h_{X_j} \otimes h_{X_j}  \rangle_{\mathrm{HS}} = \langle \Sigma_{XX}, \Sigma_{XX}  \rangle_{\mathrm{HS}}$, for $i \neq j$. This allows us to ignore all pairs of distinct indices in \eqref{eq:covariance_convergence_part_2}, and we obtain that,
	\begin{align}\label{eq:covariance_convergence_part_3}
		\frac{1}{n^2} \sum_{i=1}^n \sum_{j=1}^n \mathrm{E} \langle H_i, H_j  \rangle_{\mathrm{HS}} = \frac{1}{n^2} \sum_{i=1}^n \mathrm{E} \| H_i \|_{\mathrm{HS}}^2 = \frac{1}{n} \mathrm{E} \| H_1 \|_{\mathrm{HS}}^2.
	\end{align}
	Correspondingly,
	\begin{align*}
		\mathrm{E} \| H_1 \|_{\mathrm{HS}}^2 \leq \; & \mathrm{E} \| h_{X_1} \otimes h_{X_1} \|_{\mathrm{HS}}^2 + 2 \mathrm{E} \| h_{X_1} \otimes h_{X_1} \|_{\mathrm{HS}} \| \Sigma_{XX} \|_{\mathrm{HS}} + \| \Sigma_{XX} \|_{\mathrm{HS}}^2\\
		= \; & \mathrm{E} \| h_{X_1} \|_{\Hcal_X}^4 + 2 \mathrm{E} \| h_{X_1} \|_{\Hcal_X}^2 \| \Sigma_{XX} \|_{\mathrm{HS}} + \| \Sigma_{XX} \|_{\mathrm{HS}}^2,
	\end{align*}
	which is finite, because $\Sigma_{XX}$ is a Hilbert-Schmidt operator, and that, by Assumption~\ref{assu:fourth_moments}, $\mathrm{E} \| \kappa_{X}(\cdot, X_1) \|_{\Hcal_X}^4 = \mathrm{E} \{ \kappa_X(X_1, X_1)^2 \} < \infty $, guaranteeing that $\mathrm{E} \| h_{X_1} \|_{\Hcal_X}^4$ is finite. Together, \eqref{eq:covariance_convergence_part_2} and \eqref{eq:covariance_convergence_part_3}, along with Markov's inequality, imply that
	\begin{align*}
		\left\| \frac{1}{n} \sum_{i=1}^n (h_{X_i} \otimes h_{X_i}) - \Sigma_{XX} \right\|_{\mathrm{HS}} = \mathcal{O}_p(1/\sqrt{n}).
	\end{align*}
	
	Combining the results above, we obtain that $\| \hat{\Sigma}_{XX} - \Sigma_{XX} \|_{\mathrm{HS}}  = \mathcal{O}_p(1/\sqrt{n})$. The results for $\hat{\Sigma}_{XY}$ and $\hat{\Sigma}_{YY}$ can be established similarly. This completes the proof of Lemma \ref{lem:technical_1}. 
\end{proof}

The second auxiliary lemma shows that the inverse operators $G_{n1}^{-1}$ and $G_{n2}^{-1}$ are bounded in the operator norm, where $G_{n1}^{-1} = (\hat{\Sigma}_{XX} + \tau I)^{-1}$ and $G_{n2}^{-1} = (\Sigma_{XX} + \tau I)^{-1}$. 

\begin{lemma}\label{lem:technical_2}
	Suppose the conditions of Theorem \ref{theo:asymptotics_theorem_1} hold. Then, $\| G_{n1}^{-1} \|_{\mathrm{OP}} \leq 1/\tau$, and $\| G_{n2}^{-1} \|_{\mathrm{OP}} \leq 1/\tau$.
\end{lemma}

\begin{proof}[Proof of Lemma \ref{lem:technical_2}]
	Note that $\| G_{n2}^{-1} \|_{\mathrm{OP}} = (1/\tau) \| (\Sigma_{XX}/\tau + I)^{-1} \|_{\mathrm{OP}} \leq 1/\tau$. This relation holds because, by the positive semi-definiteness of $T = \Sigma_{XX}/\tau$, we have, for arbitrary $f \in \Hcal_X$, 
	\begin{align*}
		\| (T + I)^{-1} f \|^2_{\Hcal_X} \leq \; & \langle (T + I) (T + I)^{-1} f, (T + I)^{-1} f \rangle_{\Hcal_X} \\
		\leq \; & \| f \|_{\Hcal_X} \| (T + I)^{-1} f \|_{\Hcal_X}, 
	\end{align*}
	implying that $\| (T + I)^{-1} f \|_{\Hcal_X} \leq \| f \|_{\Hcal_X}$. Similarly, we can show that $\| G_{n1}^{-1} \|_{\mathrm{OP}} \leq 1/\tau $. This completes the proof of Lemma \ref{lem:technical_2}. 
\end{proof}

The third auxiliary lemma establishes the convergence rate for the effect of replacing the pseudo-inverse with its regularized counterpart on the population level.

\begin{lemma}\label{lem:technical_3}
	Suppose the conditions of Theorem \ref{theo:asymptotics_theorem_1} hold. Then, $\| G_{n2}^{-1} \Sigma_{XY} -  \Sigma_{XX}^\dagger  \Sigma_{XY} \|_{\mathrm{OP}} = \mathcal{O}(\tau)$.
\end{lemma}

\begin{proof}[Proof of Lemma \ref{lem:technical_3}]
	By Assumption \ref{assu:asymptotics_regularity} and Theorem 1 of \cite{douglas1966majorization}, we have $\Sigma_{XY} = \Sigma_{XX}^2 C$, for some bounded operator $C: \Hcal_X \to \Hcal_X$. This further implies that,
	\begin{align*}
		\| G_{n2}^{-1} \Sigma_{XY} -  \Sigma_{XX}^\dagger  \Sigma_{XY} \|_{\mathrm{OP}} &\leq \| (G_{n2}^{-1} \Sigma_{XX} - I ) \Sigma_{XX} \|_{\mathrm{OP}} \| C  \|_{\mathrm{OP}} \\
		& = \| -\tau G_{n2}^{-1} \Sigma_{XX} \|_{\mathrm{OP}} \| C  \|_{\mathrm{OP}}= \tau \| \tau G_{n2}^{-1} - I \|_{\mathrm{OP}} \| C  \|_{\mathrm{OP}},
	\end{align*}
	where $\| \tau G_{n2}^{-1} - I \|_{\mathrm{OP}} \leq \tau \|  G_{n2}^{-1}  \|_{\mathrm{OP}} + \| I \|_{\mathrm{OP}} \leq 2$, by Lemma \ref{lem:technical_2}. Therefore,
	\begin{align*}
		\| G_{n2}^{-1} \Sigma_{XY} -  \Sigma_{XX}^\dagger  \Sigma_{XY} \|_{\mathrm{OP}} = \mathcal{O}(\tau).
	\end{align*}
	This completes the proof of Lemma \ref{lem:technical_3}.
\end{proof}


Based on the above three auxiliary lemmas, we next prove Theorem \ref{theo:asymptotics_theorem_1}.

We first establish the closeness of $ G_{n1}^{-1} \hat{\Sigma}_{XY}$ to the operator $ G_{n2}^{-1} \Sigma_{XY}$. Note that 
\begin{align}\label{eq:first_pseudoinverse_expansion}
	\begin{split}
		& \| G_{n1}^{-1} \hat{\Sigma}_{XY} -  G_{n2}^{-1} \Sigma_{XY} \|_{\mathrm{OP}} \\
		\leq \; &  \|  G_{n1}^{-1} \|_{\mathrm{OP}} \|  \hat{\Sigma}_{XY} -  \Sigma_{XY} \|_{\mathrm{OP}} + \| (G_{n1}^{-1} - G_{n2}^{-1}) \Sigma_{XY} \|_{\mathrm{OP}}.
	\end{split}
\end{align}
To simplify \eqref{eq:first_pseudoinverse_expansion}, we observe that, by Assumption \ref{assu:asymptotics_regularity}, we have $\Sigma_{XY} = \Sigma_{XX} D$ for some bounded operator $D$. This implies that 
\begin{align*}
	\| ( G_{n1}^{-1} - G_{n2}^{-1} ) \Sigma_{XY} \|_{\mathrm{OP}} &= \| G_{n1}^{-1} ( G_{n1} - G_{n2} ) G_{n2}^{-1} \Sigma_{XY} \|_{\mathrm{OP}}\\
	& \leq \| G_{n1}^{-1} \|_{\mathrm{OP}} \| G_{n1} - G_{n2} \|_{\mathrm{OP}} \| G_{n2}^{-1} \Sigma_{XY} \|_{\mathrm{OP}} \\
	& \leq (1/\tau) \| \hat{\Sigma}_{XX} - \Sigma_{XX}  \|_{\mathrm{OP}}  \| G_{n2}^{-1} \Sigma_{XY} \|_{\mathrm{OP}} \\
	& \leq \mathcal{O}_p(1/\{\tau \sqrt{n}\}) \| (\Sigma_{XX} + \tau I)^{-1} \Sigma_{XX} \|_{\mathrm{OP}} \| D \|_{\mathrm{OP}}  \\
	&= \mathcal{O}_p(1/\{\tau \sqrt{n}\}),
\end{align*}
where the last equality holds because the largest eigenvalue of the operator $(\Sigma_{XX} + \tau I)^{-1} \Sigma_{XX}$ is bounded from above by one.

Therefore, together with Lemmas \ref{lem:technical_1} and \ref{lem:technical_2}, we have that
\begin{align}\label{eq:first_pseudoinverse_result}
	\| G_{n1}^{-1} \hat{\Sigma}_{XY} -  G_{n2}^{-1} \Sigma_{XY} \|_{\mathrm{OP}} = \mathcal{O}_p(1/\{\tau \sqrt{n}\}).
\end{align}
Combining \eqref{eq:first_pseudoinverse_result} with Lemma \ref{lem:technical_3} leads to
\begin{align*}
	\| G_{n1}^{-1} \Sigma_{XY} -  \Sigma_{XX}^\dagger  \Sigma_{XY} \|_{\mathrm{OP}} = \mathcal{O}_p( \tau + 1/\{\tau \sqrt{n}\} ).
\end{align*}
The sample convergence of $(\hat{\Sigma}_{YY} + \tau I)^{-1} \hat{\Sigma}_{YX}$ can be established similarly. This completes the proof of Theorem \ref{theo:asymptotics_theorem_1}. 

\eop

\bigskip
\medskip
\noindent
\textbf{Proof of Theorem \ref{theo:asymptotics_theorem_2}}:
Denote $\mathbb{I}_{ik} = \mathbb{I}(Y_i = k)$, $N_k = \sum_{i=1}^n \mathbb{I}_{ik}$, and $h_{X_i} = \kappa_X(\cdot, X_i) - \mu_X$. We have that
\begin{align*}
	\hat{\gamma}_{X \mid k} = \frac{1}{N_k} \sum_{i=1}^n \mathbb{I}_{ik} h_{X_i} - \frac{1}{n} \sum_{i=1}^n h_{X_i} =\frac{1}{N_k} \sum_{i=1}^n \mathbb{I}_{ik} h_{X_i} - \bar{h}_n.
\end{align*}
Consequently,
\begin{align}\label{eq:conditional_mean_deviation}
	\begin{split}
		\mathrm{E} \| \hat{\gamma}_{X \mid k} - \gamma_{X \mid k} \|^2_{\Hcal_X} = \; & \mathrm{E} \left\| \frac{1}{N_k} \sum_{i=1}^n \mathbb{I}_{ik} h_{X_i} - \gamma_{X \mid k}  \right\|^2_{\Hcal_X} + \mathrm{E} \| \bar{h}_n \|_{\Hcal_X}^2 \\
		& - 2 \mathrm{E} \left\langle \frac{1}{N_k} \sum_{i=1}^n \mathbb{I}_{ik} h_{X_i} - \gamma_{X \mid k} , \bar{h}_n \right\rangle_{\Hcal_X}.
	\end{split}
\end{align}

The first term of the right-hand-side of \eqref{eq:conditional_mean_deviation} is equal to,
\begin{align}\label{eq:discrete_term_1}
	\sum_{i = 1}^n \sum_{j = 1}^n \mathrm{E} \left( \frac{1}{N_k^2} \mathbb{I}_{ik} \mathbb{I}_{jk} \langle h_{X_i} - \gamma_{X \mid k} , h_{X_j} - \gamma_{X \mid k} \rangle_{\Hcal_X} \right).
\end{align}
Denote $v_{ij} = \langle h_{X_i} - \gamma_{X \mid k} , h_{X_j} - \gamma_{X \mid k} \rangle_{\Hcal_X} $. Then, for any distinct $i \neq j$, conditioning on the values of the corresponding indicators implies the summand in \eqref{eq:discrete_term_1} equals 
\begin{align*}
	\mathrm{E} \left( \frac{1}{N_k^2} v_{ij} \bigm| \mathbb{I}_{ik} \mathbb{I}_{jk} = 1 \right) \pi_k^2 = \mathrm{E} \left( \frac{1}{N_k^2} \bigm| \mathbb{I}_{ik} \mathbb{I}_{jk} = 1 \right) \mathrm{E} \left( v_{ij} \bigm| \mathbb{I}_{ik} \mathbb{I}_{jk} = 1 \right) \pi_k^2,
\end{align*}
where $\pi_k = P(Y = k)$. Using a similar argument as in \eqref{eq:parseval_argument} and the definition of $\gamma_{X \mid k}$, we have that $\mathrm{E} \left( v_{ij} \bigm| \mathbb{I}_{ik} \mathbb{I}_{jk} = 1 \right) = 0$, implying that \eqref{eq:discrete_term_1} is further equal to
\begin{align}\label{eq:discrete_term_2}
	\begin{split}
		& \sum_{i = 1}^n \mathrm{E} \left( \frac{1}{N_k^2} \mathbb{I}_{ik} \| h_{X_i} - \gamma_{X \mid k} \|^2_{\Hcal_X} \right) \\
		= \; & n \pi_k \mathrm{E} \left( \frac{1}{N_k^2} \mid \mathbb{I}_{1k} = 1 \right) \mathrm{E} \left( \| h_{X_1} - \gamma_{X \mid k} \|^2_{\Hcal_X} \mid \mathbb{I}_{1k} = 1 \right).
	\end{split}
\end{align}
Correspondingly, we have that
\begin{align*}
	\pi_k \mathrm{E} ( \| h_{X_1} - \gamma_{X \mid k} \|^2_{\Hcal_X} \mid \mathbb{I}_{1k} = 1 ) &=  \pi_k \{ \mathrm{E} ( \| h_{X_1} \|^2_{\Hcal_X} \mid \mathbb{I}_{1k} = 1 ) - \| \gamma_{X \mid k} \|_{\Hcal_X}^2 \} \\
	& \leq \mathrm{E} ( \| h_{X_1} \|^2_{\Hcal_X} ) < \infty.
\end{align*}
We next tackle the term $\mathrm{E} ( 1/N_k^2 \mid \mathbb{I}_{1k} = 1 )$. Conditioning on $\{ \mathbb{I}_{1k} = 1 \}$, the distribution of $(n - 1)^2/N_k^2$ is that of $(n - 1)^2/(B + 1)^2$, where $B$ follows a $\mathrm{Binomial}(n - 1, \pi_k)$ distribution. By Theorem 1 of \cite{shi2010note}, the expected value of $(n - 1)^\alpha/(B + c)^\alpha$ is of the order $\mathcal{O}(1)$ for any $c, \alpha > 0$. This further implies that \eqref{eq:discrete_term_2} and, consequently, the first term on the right-hand-side of \eqref{eq:conditional_mean_deviation}, is of the order $\mathcal{O}(1/n)$.

The second term of the right-hand-side of \eqref{eq:conditional_mean_deviation}, as shown in Theorem \ref{theo:asymptotics_theorem_1} and by Assumption \ref{assu:fourth_moments}, is of the order $\mathcal{O}(1/n)$. 

The third term of the right-hand-side of \eqref{eq:conditional_mean_deviation} can be expressed as
\begin{align}\label{eq:discrete_term_3}
	\begin{split}
		-\frac{2}{n} \sum_{i = 1}^n \sum_{j = 1}^n \mathrm{E} \left( \frac{1}{N_k} \mathbb{I}_{ik} \langle h_{X_i} - \gamma_{X \mid k} , h_{X_j} \rangle \right).
	\end{split}
\end{align}
Conditioning on the indicators, we can rewrite the contribution of any index pair of $i \neq j$ to the sum as
\begin{align*}
	\sum_{\ell = 1}^K \pi_k \pi_\ell \mathrm{E} ( 1 /N_k \mid \mathbb{I}_{ik} \mathbb{I}_{j\ell} = 1 ) \mathrm{E} (\langle h_{X_i} - \gamma_{X \mid k} , h_{X_j} \rangle \mid \mathbb{I}_{ik} \mathbb{I}_{j\ell} = 1 ),
\end{align*}
where the final expected value is equal to zero, following a similar argument to \eqref{eq:parseval_argument}. Therefore, only the index pairs of $i = j$ contribute to the sum \eqref{eq:discrete_term_3}, which then takes the form,
\begin{align*}
	& -\frac{2}{n} \sum_{i = 1}^n \mathrm{E} \left( \frac{1}{N_k} \mathbb{I}_{ik} \langle h_{X_i} - \gamma_{X \mid k} , h_{X_i} \rangle \right) \\
	=& -2 \pi_k \mathrm{E} (1/N_k \mid \mathbb{I}_{ik} = 1) \mathrm{E} (\langle h_{X_i} - \gamma_{X \mid k} , h_{X_i} \rangle \mid \mathbb{I}_{ik} = 1 ).
\end{align*}
Now, $\mathrm{E} (1/N_k \mid \mathbb{I}_{ik} = 1) = \mathcal{O}(1/n)$ by Theorem 1 of \cite{shi2010note}. Moreover, we can show that the term $\pi_k \mathrm{E} (\langle h_{X_i} - \gamma_{X \mid k} , h_{X_i} \rangle \mid \mathbb{I}_{ik} = 1 )$ is of the order $\mathcal{O}(1)$. Consequently, the third term on the right-hand-side of \eqref{eq:conditional_mean_deviation} is of the order $\mathcal{O}(1/n)$. 

Applying the Markov's inequality obtains that $\| \hat{\gamma}_{X \mid k} - \gamma_{X \mid k} \|_{\Hcal_X} = \mathcal{O}_p(1/\sqrt{n})$.

We next show that $\sum_{k = 1}^K (N_k/n) (\hat{\gamma}_{X \mid k} \otimes \hat{\gamma}_{X \mid k})$ converges in the Hilbert-Schmidt norm to $\Gamma_{XX \mid Y} = \sum_{k=1}^K \pi_k (\gamma_{X \mid k} \otimes \gamma_{X \mid k})$.
Note that the norm of the difference $ \sum_{k = 1}^K (N_k/n) (\hat{\gamma}_{X \mid k} \otimes \hat{\gamma}_{X \mid k}) - \sum_{k = 1}^K \pi_k (\gamma_{X \mid k} \otimes \gamma_{X \mid k})$ is upper-bounded by
\vspace{-0.01in}
\begin{align}\label{eq:expected_tensor_product_convergence}
	& \left\| \sum_{k = 1}^K (N_k/n) (\hat{\gamma}_{X \mid k} \otimes \hat{\gamma}_{X \mid k}) - \sum_{k = 1}^K \pi_k (\gamma_{X \mid k} \otimes \gamma_{X \mid k}) \right\|_{\mathrm{HS}} \nonumber \\
	\leq \; & \sum_{k = 1}^K \| (N_k/n) (\hat{\gamma}_{X \mid k} \otimes \hat{\gamma}_{X \mid k}) - \pi_k  (\gamma_{X \mid k} \otimes \gamma_{X \mid k}) \|_{\mathrm{HS}} \nonumber \\
	\leq& \sum_{k = 1}^K \{ (N_k/n) \| (\hat{\gamma}_{X \mid k} \otimes \hat{\gamma}_{X \mid k}) - (\gamma_{X \mid k} \otimes \gamma_{X \mid k}) \|_{\mathrm{HS}} + |(N_k/n) - \pi_k| \| \gamma_{X \mid k} \otimes \gamma_{X \mid k} \|_{\mathrm{HS}} \} \nonumber\\
	\leq \; & \sum_{k = 1}^K \{ (N_k/n) \|  ( \hat{\gamma}_{X \mid k} - \gamma_{X \mid k}) \otimes \hat{\gamma}_{X \mid k} \|_{\mathrm{HS}} \nonumber\\
	& + (N_k/n) \| \{ \gamma_{X \mid k} \otimes ( \hat{\gamma}_{X \mid k} - \gamma_{X \mid k} ) \} \|_{\mathrm{HS}}
	+ |(N_k/n) - \pi_k| \| \gamma_{X \mid k} \otimes \gamma_{X \mid k} \|_{\mathrm{HS}} \}.
\end{align}
Note that $\| ( \hat{\gamma}_{X \mid k} - \gamma_{X \mid k}) \otimes \hat{\gamma}_{X \mid k} \|_{\mathrm{HS}}^2 = \| \hat{\gamma}_{X \mid k} - \gamma_{X \mid k} \|_{\Hcal_X}^2 \| \hat{\gamma}_{X \mid k} \|^2_{\Hcal_X} = \mathcal{O}_p(1/n)$, because $ \|  \hat{\gamma}_{X \mid k} \|^2_{\Hcal_X}$ converges to the finite constant $ \|  \gamma_{X \mid k} \|^2_{\Hcal_X}$ as $n \rightarrow \infty$.  Similarly, we can show that $\| \gamma_{X \mid k} \otimes ( \hat{\gamma}_{X \mid k} - \gamma_{X \mid k} ) \|_{\mathrm{HS}}^2 = \mathcal{O}_p(1/n)$, and that $|(N_k/n) - \pi_k| = \mathcal{O}_p(1/\sqrt{n})$, which follows from the standard Central Limit Theorem. Substituting all terms with their rates in  \eqref{eq:expected_tensor_product_convergence}, we have that $\| \sum_{k = 1}^K (N_k/n)  (\hat{\gamma}_{X \mid k} \otimes \hat{\gamma}_{X \mid k}) - \sum_{k = 1}^K \pi_k (\gamma_{X \mid k} \otimes \gamma_{X \mid k}) \|_{\mathrm{HS}} = \mathcal{O}_p(1/\sqrt{n})$.

We further employ the proof of Theorem \ref{theo:asymptotics_theorem_1} to deal with the inclusion of the pseudo-inverses. This completes the proof of Theorem \ref{theo:asymptotics_theorem_2}. 
\eop

\bibhang=1.7pc
\bibsep=2pt
\fontsize{9}{14pt plus.8pt minus .6pt}\selectfont
\renewcommand\bibname{\large \bf References}
\expandafter\ifx\csname
natexlab\endcsname\relax\def\natexlab#1{#1}\fi
\expandafter\ifx\csname url\endcsname\relax
  \def\url#1{\texttt{#1}}\fi
\expandafter\ifx\csname urlprefix\endcsname\relax\def\urlprefix{URL}\fi

\bibliographystyle{chicago}      
\bibliography{ref-msdr}





\end{document}